\documentclass[12pt,a4paper,reqno]{amsart}

\usepackage[margin = 1in]{geometry}
\usepackage{a4wide}
\usepackage{amssymb,amsmath,amsthm,mathrsfs}
\usepackage{graphicx}
\graphicspath{{/Users/mcarney/Documents/Functional_Extremes_Paper/22-03-2023/final_manuscript_revisions/images_manuscript}}
\usepackage{verbatim}
\usepackage{caption}
\usepackage{float}
\usepackage{colortbl}
\usepackage{enumerate}
\usepackage{upref}
\usepackage{tikz}
\usepackage{pgfplots}
\pgfplotsset{compat=1.11}
\usepgfplotslibrary{fillbetween}
\usetikzlibrary{intersections}
\usetikzlibrary{arrows,backgrounds,patterns}
\pgfdeclarelayer{bkgr}
\pgfdeclarelayer{bkgrd}
\pgfsetlayers{bkgrd,bkgr,main}
\usetikzlibrary{arrows}
\usetikzlibrary{decorations.pathmorphing}
\usetikzlibrary{decorations.pathreplacing}
\usepackage[T2A,T1]{fontenc}
\DeclareSymbolFont{cyrillic}{T2A}{cmr}{m}{n}
\DeclareMathSymbol{\D}{\mathalpha}{cyrillic}{196}

\usepackage{bbm}
\usepackage[bookmarks=false]{hyperref}
\theoremstyle{plain}
\newtheorem{theorem}{Theorem}[section]
\newtheorem{lemma}[theorem]{Lemma}

\newtheorem{proposition}[theorem]{Proposition}

\theoremstyle{definition}

\newtheorem{example}{Example}[section]

\theoremstyle{remark}
\newtheorem{remark}[theorem]{Remark}

\usepackage{enumitem}
\makeatletter
\def\namedlabel#1#2{\begingroup
   #2%
 \def\@currentlabel{#2}%
   \phantomsection\label{#1}\endgroup
}

\newcommand{\set}[1]{\left\{#1\right\}}
\newcommand{\abs}[1]{\left|#1\right|}
\usepackage[T2A,T1]{fontenc}
\makeatletter
\newcommand*{\math@version@bold}{bold}
\DeclareMathOperator\DD{
  \textrm{%
    \usefont{T2A}{cmr}{\ifx\math@version\math@version@bold bx\else m\fi}{n}%
    \CYRD
  }%
}
\makeatother
\def\R{\ensuremath{\mathbb R}}
\def\T{\ensuremath{\mathbb T}}
\def\N{\ensuremath{\mathbb N}}

\def\P{\ensuremath{\mathcal P}}

\numberwithin{equation}{section}

\newcommand{\bg}[3]{\bigg#1 #2 \bigg#3}

\newcommand{\gev}[1]{\exp {- \bigg[ 1 + \xi\, \bigg( \frac{#1 - \mu}{\sigma}\, \bigg)\, \bigg]^{-\frac{1}{\xi}} }}

\begin{document}

\title[\hfill\protect\parbox{0.975\linewidth}{Runs of extremes. }]{Runs of extremes of observables   on dynamical systems
and applications.}

\author[M. Carney]{Meagan Carney}
\address{Meagan Carney\\ Department of Mathematics\\
University of Houston\\
Houston\\
TX 77204\\
USA} \email{m.carney@uq.edu.au}
\urladdr{https://smp.uq.edu.au/profile/11136/meagan-carney}

\author[M. Holland]{Mark Holland}
\address{Mark Holland\\ Department of Mathematics \& Statistics \\
Harrison Building (327)\\
North Park Road\\
Exeter, EX4 4QF\\ UK} \email{M.P.Holland@exeter.ac.uk}
\urladdr{http://empslocal.ex.ac.uk/people/staff/mph204/}

\author[M. Nicol]{Matthew Nicol}
\address{Matthew Nicol\\ Department of Mathematics\\
University of Houston\\
Houston\\
TX 77204\\
USA} \email{nicol@math.uh.edu}
\urladdr{http://www.math.uh.edu/~nicol/}

\author[P. Tran]{Phuong Tran}
\address{Phuong Tran\\ Department of Mathematics\\
University of Houston\\
Houston\\
TX 77204\\
USA} \email{ptran8791@gmail.com}

\thanks{MN thanks  the NSF  for support from NSF-DMS 2009923.}

\begin{abstract}

We use extreme value theory to estimate the probability of successive exceedances of a threshold value of a time-series of an observable on 
several classes of chaotic dynamical systems. The observables have either a Fr\'echet (fat-tailed) or Weibull (bounded) distribution. 
The motivation for this work was  to give  estimates of the probabilities of sustained periods of weather anomalies such as heat-waves, cold spells or prolonged periods of rainfall in climate models. Our predictions are borne out by numerical simulations  and also analysis of rainfall and temperature data.  

\end{abstract}

\maketitle

\section{Introduction}

The impact of successive extreme weather events on populations has become a significant topic of discussion in climate literature. Recent notable examples include the 2019 global heat-waves, which had severe impacts in Australia, Europe, and the United States, and the 2022 heat-wave in western Europe resulting in a nationally declared emergency in the United Kingdom. Other examples include the 2021-22 floods in Eastern Australia, the 2021 flood in the United Kingdom and Western Europe and the 2017 flood in Texas, all attributed to prolonged heavy rainfall. Understanding the returns of successive extreme weather events like these is crucial for preparing and mitigating disastrous effects. For instance, fire retardant materials can be used to prevent wildfires during heat-waves, and earlier evacuation of populations from high-risk flood zones can be initiated prior to episodes of prolonged rainfall.

Accurately predicting the likelihood of weather anomalies, such as heat-waves, cold spells, or prolonged periods of rainfall, is a challenging problem. In this paper, we explore the use of extreme value theory (EVT) to estimate the probability of successive exceedances of a threshold of a time-series of observations on a variety of uniformly hyperbolic dynamical systems. The study of chaotic dynamical systems provides insights into the statistical behavior of complex systems, such as climate models. We test some of the predictions we derive from uniformly hyperbolic systems on climate data and climate models.

Provided the time-series of an observable $\phi$, namely $(\phi\circ T^j )_{j\ge0}$, follow the assumptions outlined in \cite{LLR}, an extreme value law holds for the maxima of the time-series, given by $M_j = \max\{\phi\dots,(\phi\circ T^j )\}$, which guarantees a limiting probability distribution for the sequence of maxima $(M_j)$. The traditional measure of successive extremes is the extremal index, a parameter which appears in the limiting distribution of the extreme value law and measures the expected number of over-threshold exceedances observed in a short block of the time-series. 

For appropriate scaling choices, the exact limiting distribution is the generalized extreme value distribution (GEV), $G(\xi,\mu,\sigma)$, with shape $\xi$, location $\mu$ and scale $\sigma$ parameters. In practice, the \textit{block maximum approach} is commonly employed which involves dividing the time-series into blocks of a fixed length and performing some likelihood fitting of the parameters of the GEV using the sequence of block maxima. By nature of the block maximum approach, the extremal index is normalized to 1 so that all information on the successive properties of extremes of the data becomes lost. However, a careful choice of observable can preserve some of the properties of successive extremes in this setting. 

In climate applications, successive extremes are often thought of in two ways: a high average in a short run, for example, a high daily temperature average over 7 days; or successive over-threshold values occurring in a short run, for example, a high daily temperature observed every day for 7 days. Motivated by~\cite{Galfi_Lucarini_Wouters} and these definitions, we investigate the time-averages of observables over a time-window of integer length $k$,  $Y (x)=\frac{1}{k}\sum_{j=0}^{k-1} \phi (T^j x)$, as well the exceedance function $Y (x) =\min \{\phi (x), \ldots, \phi (T^{k-1} x)\}$.


This paper's focus is on whether the scaling constants in the time-series of $k$-exceedances or time-averages can be derived from that of the original time-series. We establish scaling constants in some interesting scenarios for hyperbolic systems. Specifically, we consider an observable $\phi: X\rightarrow\mathbb{R}$ on an ergodic dynamical system $(T,X,\mu)$, which is maximized on an invariant repelling set $S$. The prolonged bouts of extreme weather events are associated with, for example, a weather system being in a quasi-invariant state of extreme rainfall or temperature. This is broadly modeled by the time-series of an observable maximized on an invariant set in a chaotic dynamical system. 

We begin our investigation by establishing a general lemma relating the parameters of the GEV for the $k$-exceedances and time-averages coming from any system, with a Fr\'{e}chet or Weibull limiting distribution, satisfying the conditions of \cite{LLR}. This general result is then used to establish more concrete theorems for climate-relevant observables taken on hyperbolic systems. We end with a detailed illustration of the applications of these lemmas and theorems to numerically model successive extremes of time-series data coming from: certain hyperbolic dynamical systems; temperature data simulated from the general circulation model, PUMA; and real-world temperature and precipitation data taken from weather stations throughout Germany. Finally, we illustrate that our scaling results can be applied to obtain more accurate numerical estimates of the GEV parameters for maxima of $k$-exceedance and time-averaged functionals of increasing $k$ window lengths than those of traditional maximum likelihood estimation.

This research was motivated in part by the interesting paper~\cite{Galfi_Lucarini_Wouters} in which the authors analyzed the time and spatial averages of daily temperature anomalies over a large number of days (heat-waves or cold spells) via a large deviation approach.
Large deviation theory was  used to estimate probabilities associated to heat-wave or cold-spell occurrences. The averaging period ranged roughly from 10 to 40  days. The main testing ground was the PUMA climate model of mid-atlantic latitudes, 30 degrees to 60 degrees. 
The important question they addressed is how to estimate the probability of very rare events, such as long runs of extreme weather,  from a limited time-series of recordings.
Their predictions were favorably compared to those estimated by standard extreme value theory (Peaks over threshold model leading to a generalized Pareto distribution (GPD)) applied to long simulations of the PUMA model. A related approach  based on a large deviation algorithm combined with an  importance sampling technique was given in~\cite{Ragone_Wouters_Bouchet}. 

Our dynamical models and numerical results suggest that in certain settings we do not need to average over long time windows to ensure the applicability of extreme value theory and in some sense we may  estimate rare prolonged anomalies from the original time-series of data. The advantage of our method is that while data on, for example, prolonged spells of rainfall is sparse, the simple scalings we find enable us to use all available daily maxima data to estimate the probabilities of prolonged rain  of a certain duration.

\section{Background on EVT for dynamical systems.}

Suppose $\phi:X\to \R$ is an observable on an ergodic dynamical system $(T,X,\mu)$. Extreme value theory (EVT)
 considers  distributional limits of the maxima process
 \[
	M_n=\max \{ \phi,\ldots, \phi\circ T^{n-1} \}
\]
under suitable scaling constants $a_n (M_n-b_n)$. 
	
P. Collet~\cite{Collet} introduced techniques from EVT  to establish hitting and return time statistics to `generic' non-recurrent points in 
some one-dimensional non-uniformly hyperbolic dynamical systems. He considered the observable $\phi (x)=-\ln d(x,x_0)$, where $x_0$
is non-recurrent, in particular not periodic. In this setting EVT is related to hitting and 
return time statistics since $M_n$ increases as orbits make closer approaches to  $x_0$.
Following Collet's work there have been many papers
developing EVT for  hyperbolic dynamical systems. The theory  is well-developed for
$\phi$  maximized at a point $x_0$ and  a function of
distance, as in Collet's case of $\phi (x)=-\ln d(x,x_0)$.  The paper~\cite{CHN} extended some of this
work to the more general setting of observables maximized on submanifolds in phase space, while ~\cite{FFRS} considered observables 
maximized on Cantor sets. In a series of influential papers Freitas et al~\cite{FFT1,FFT5} adapted techniques from~\cite{LLR} 
to the context of deterministic dynamical systems and elucidated the close relation between extremes, hitting time statistics and Poisson processes. For these and other developments we refer to the book~\cite{Extremes_Book} for more details.

More generally suppose  $(X_n)$ is a  stationary process with probability distribution function $F_X(u):=\mu(X\leq u).$ 
When determining the distributional limit of 
\[
M_n =\max\{ X_1, \ldots, X_n\}
\]
it is well-known that to  find the appropriate scaling, given  $\tau\in\mathbb{R}$, one should define  $u_n(\tau)$ as  a sequence satisfying $n\mu(X_0>u_{n}(\tau))\to\tau$, as $n\to\infty$. We say that $(X_n)$ satisfies an extreme value law if 
\begin{equation}\label{eq.ev-law}
\mu(M_n\leq u_n(\tau))\to e^{-\theta\tau}
\end{equation}
for some $\theta\in(0,1]$.
The number  $\theta$ is called the extremal index and $\frac{1}{\theta}$ roughly measures the average cluster size of exceedances given that one exceedance has occurred.
When $(X_n)$ is  iid and has a regularly varying tail it can be shown that under the scaling $u_n$ there exists an extreme value law (EVL) and $\theta=1$. 
If $(X_n)$ is only assumed stationary rather than iid then the existence of an EVL with $\theta=1$  has been shown provided two conditions hold:  (1)  $D(u_n)$ (a mixing condition) and;  (2) $D'(u_n)$ (a non-recurrence condition).   
Thus in the case of an observable on a dynamical system if the time series of observations $X_n=\phi\circ T^n$  satisfy $D(u_n)$ and $D'(u_n)$ (or some variation thereof) then an EVL holds with $\theta=1$.

We will use two conditions $\DD(u_n)$ and $\DD^{'}_q (u_n)$, adapted to the dynamical setting, introduced in the work~\cite{FFT5} that imply a non-degenerate EVL in the 
case the EI  $\theta \not =1$ and also give the value of  $\theta$.

\subsection{Using EVT to estimate the probability of rare events.}

For statistical estimation and fitting schemes such as block maxima or peak over thresholds methods~\cite{Embrechts}, it is desirable to get a limit along linear sequences of the form $u_n(y)=y/a_n+b_n$. 
Here the emphasis is changed and the sequence $u_n(y)$ is now required to be linear in $y$. 


 R.  Fisher and L. Tippett~\cite{Fisher_Tippett} showed:   If $\{X_n\}_{n\in \N}$ a sequence of iid random variables 
  and there exists linear normalizing sequences of constant $\set{a_n>0}_{n\in\N}$ and $\set{b_n}_{n\in\N}$ such that
	        \[ \P\bg{(}{\frac{M_n - b_n}{a_n}\leq y}{)} \to G(y)\qquad\text{as } n\to\infty
	        \]
	        where G is a non-degenerate distribution function
	         then $G(y)=e^{-\tau(y)}$, where $\tau(y)$ is one of the following three types (for some $\beta,\gamma>0$), (up to a scale $\sigma$ and location $\mu$ shift $y\to \frac{y-\mu}{\sigma}$):
	        \begin{itemize}
	            \item [(1)] $\tau (y)=e^{-y}$ for $y\in\R$; (Gumbel) \\
	            \item [(2)] $\tau (y)=y^{-\beta}$ for $y>0$ (Fr\'echet)\\
	            \item [(3)] $\tau (y)=(-y)^{\gamma}$ for $y\le 0$ (Weibull).
	        \end{itemize}

        The three types may be combined in a unified model called the \emph{Generalized Extreme Value} (GEV) distribution
        \[ 
        G(y) = \gev{y}
        \]
        defined on $\set{y:\, 1+ \xi\left(\frac{y -\mu}{\sigma}\right) > 0}$, where $\xi\in\R$ is the shape parameter $\mu\in\R$ is the location parameter,  and $\sigma > 0$ the scale parameter. We have the following classification:
        
        \begin{itemize}
            \item[(1)]  When $\xi=0$, Type I - The Gumbel distibution
            $$G(y) = \exp\left\{-e^{-(\frac{y-\mu}{\sigma})}\right\} \quad ,y\in\R; $$
            \item[(2)]  $\xi>0$ Type II - The Fr\'{e}chet distibution;
            \item[(3)] $\xi<0$ Type III - The  Weibull distribution.
        \end{itemize}

The Gumbel distribution is unbounded, while the Weibull is bounded. The Fr\'echet distribution is bounded below but not above and has `fat tails', $1-F(x)=l(x)x^{-\xi}$ where $l(x)$ is a slowly varying function. An important point is that 
numerical fitting schemes for the GEV distribution are renormalized under place and scale transformations so that the extremal index (EI) is 1 \cite[Theorem 5.2]{Coles}. Hence
the EI is not given by the GEV but subsumed into its parameter estimation.  A technique to estimate the EI has been given by S\"uveges~\cite{Suveges}.
The goal of this work is to determine if there is a relation between $\xi$, $\sigma$ and $\mu$  for the GEV for $\phi$ and the corresponding 
parameters for $Y$, both in the case of $k$-exceedances and $k$-averages.

  We now describe how the GEV model is used to estimate return levels. Suppose we have a time-series of length $mn$ observations.
  For example daily rainfall over a year  gives  $n=365$ daily readings and suppose and we have $m=100$ years of such daily measurements.
    Denote $M_{n,i}=\max\{X_i,\cdots,X_{n+i-1}\}$, $i=1,...,m$.
       We block the data  into sequences of observations of length $ n$
        generating a series of block maxima, $M_{n,1},\dots,M_{n,m}$ to which the GEV distribution can be fitted. This assumes $n$ is large enough to ensure convergence of $M_{n,i}$ to $G$ in distribution for each $i$.   Thus we have $m$ samples from a GEV distribution determining the distribution of maximal daily rainfall in a year. We may
        then numerically fit  a GEV, via method of moments or method of maximum likelihood,  and estimate the shape $\xi$, location $\mu$ and scale $\sigma$ parameters. 
       The quantiles of the distribution of the annual maximum daily rainfall are obtained  by inverting the distribution $G$ we obtain:
            \[
            z_p = 
            \begin{cases}
            \mu - \frac{\sigma}{\xi}\bg({1-[-\log(1-p)]^{-\xi}}), & \xi\neq 0 \\
            \mu - \sigma\log[-\log(1-p)], & \xi=0
            \end{cases}
            \]
						
						where $G(z_p)=1-p$ and $z_p$ is the return level associated with the return period $\frac{1}{p}$. Thus, the level $z_p$ is expected to be exceeded on average once every $\frac{1}{p}$ years.

\section{Sufficient conditions for EVT for dynamical systems.}

We let  $(T,X,\mu)$ be an ergodic dynamical system. Here $T$ is a uniformly hyperbolic measure-preserving map of a Riemannian  metric space $(X,\mu)$ equipped with a probability measure $\mu$ which is equivalent to Lebesgue measure. We let $\phi: X\to \R$ be an observable 
which is a function of Euclidean distance and  maximized on  a set $S$.
In this section we consider the extreme value theory of $k$-exceedances and 
time-averages of the observable. More precisely we consider the EVT of the derived time-series $(Y\circ T^i)$ where 
$Y=\min \{\phi, \phi\circ T, \ldots, \phi\circ T^{k-1}\}$ in the case of $k$-exceedances or $Y=\frac{1}{k} \sum_{j=0}^{k-1}\phi\circ T^j$ for time-averages.

In this paper $S$ is taken to be  either a generic non-recurrent point $x_0\in X$; a fixed or periodic point $x_0$; or a  line or curve, invariant or not. The case of a periodic orbit of minimal period $q$ reduces to the case of a fixed point for $T^q$, so we only discuss the case of fixed points. 
We will consider three classes of uniformly hyperbolic dynamical system, namely: uniformly expanding maps of the interval; hyperbolic toral automorphisms; and coupled expanding maps.  We aim not for the greatest generality of dynamical system, as the statements and proofs
soon become tiresomely technical. We wish rather for simple illustrative models from chaotic dynamics
which hopefully shed light on general principles and the behavior of complex physical models.
The extreme value theory for observables  $\phi: X\to \R$   maximized on  such sets  $S$ for 
these systems, among others, was investigated  in~\cite{CHN}. A motivation of~\cite{CHN} was to extend the theory of
EVT for dynamical systems beyond that developed for observables maximized at points to more 
relevant observables to applications. One goal of this paper is to extend this investigation to observables such as $k$-exceedances
and time-averages.  Another goal is to allow predictions  of the GEV for $k$-exceedances
and time averages from  the GEV of the original time series, both in dynamical models and climate data.

\subsection{Conditions for EVT laws.}

We will use two conditions $\DD_q(u_n)$ and $\DD^{'}_q (u_n)$, adapted to the dynamical setting, introduced in the work~\cite{FFT5} that show an extreme 
value distribution holds and  also allows a computation of the extremal index in the case $\theta \not =1$. We change slightly the notation of~\cite{FFT5} and
 write $\DD_{q} (u_n)$ rather than $\DD )u_n)$ to highlight the role of  $q$.  In~\cite{CHN} techniques were developed to verify these conditions in the setting of an observable maximized on an invariant set for the classes of systems we consider.

Let $X_n=\phi\circ T^n$ and define 
\[
A_n^{(q)} :=\lbrace X_0>u_n, X_1 \le u_n,\ldots, X_q\leq  u_n\rbrace.
\]
For $s,l \in \mathbb{N}$ and a set $B\subset M$, let
\[
\mathscr{W}_{s,l}(B)=\bigcap_{i=s}^{s+l-1} T^{-i}(B^c).
\]
Recall that $u_n\equiv u_n(\tau)$ is a sequence satisfying $\lim_{n\to \infty} n\mu (\phi > u_n (\tau) ) =\tau$.

\noindent {\textbf{ Condition $\DD(u_n)$}}: We say that $\DD_q(u_n)$ holds for the sequence $X_0,X_1,\ldots$ if, for every  $\ell,t,n\in \mathbb{N}$
\[
\left|\mu\left(A_n^{(q)}\cap
\mathscr{W}_{t,\ell}\left(A_n^{(q)}\right) \right)-\mu\left(A_n^{(q)}\right)
\mu\left( \mathscr{W}_{0,\ell}\left(A_n^{(q)}\right)\right)\right|\leq \gamma(q,n,t),
\]
where $\gamma(q,n,t)$ is decreasing in $t$ and  there exists a sequence $(t_n)_{n\in \mathbb{N}}$ such that $t_n=o(n)$ and $n\gamma(q,n,t_n)\to0$ when $n\rightarrow\infty$.

We consider the sequence $(t_n)_{n\in\N}$ given by condition $\DD_q(u_n)$ and let $(k_n)_{n\in\N}$ be another sequence of integers such that as $n\to\infty$,
\[
k_n\to\infty\quad \mbox{and}\quad  k_n t_n = o(n).
\]

\noindent {\textbf{ Condition $\DD'_q(u_n)$}}:  
We say that $\DD'_q(u_n)$ holds for the sequence $X_0,X_1,\ldots$ if there exists a sequence $(k_n)_{n\in\N}$ as above  and such that
\[
\lim_{n\rightarrow\infty}\,n\sum_{j=q+1}^{\lfloor n/k_n\rfloor}\mu\left( A_n^{(q)}\cap T^{-j}\left(A_n^{(q)}\right)\right)=0.
\]
\begin{proposition}[~\cite{FFT5}]\label{prop.ei-comp}
Let $M_n=\max\{X_1,X_2, \ldots, X_n\}$.
Suppose that conditions $\DD'_q (u_n)$ and $\DD_q(u_n)$ hold and that the limit
\[
\theta=\lim_{n\to\infty}\theta_n=\lim_{n\to\infty}\frac{\mu(A^{(q)}_n)}{\mu(U_n)},
\]
exists.
Then
\[
\mu (M_n \le u_n (\tau))\to e^{-\theta\tau}.
\] 
\end{proposition}
Thus Proposition \ref{prop.ei-comp} above gives us a route to compute the extremal index $\theta$ for a given dynamical system and observable function, provided the conditions 
$\DD'_q (u_n)$ and $\DD_q(u_n)$ hold.

\subsection{Uniformly hyperbolic models.}

 We now briefly describe  three classes of hyperbolic dynamical system. In each there is a notion of expansion transverse
 to the invariant set which is denoted by a parameter $\lambda$  in Theorem~\ref{theorem_invariant}. We specify what $\lambda$ is 
 in each of the three settings $(A)$, $(B)$ and $(C)$ described below.

\subsubsection{(A)  Piecewise $C^2$ uniformly expanding maps of the interval. }

$T : [0,1] \to [0,1]$ is called a $C^2$ piecewise expanding map if there there is a finite partition $0=a_1<a_2<\ldots < a_m=1$
of the interval such that 
$T$ is $C^2$ on the interior of partition elements and there exists $\kappa > 1$ such that for $|DT (x)|>\kappa$ for any $x\in (a_i,a_{i+1})$, $i=1,\ldots, m-1$.
One of the simplest examples of such a  map is the doubling map which has form $T(x) = 2 x $ (mod 1). Such maps have an invariant measure $\mu$ equivalent to Lebesgue and are exponentially mixing in the sense that there exists $\rho\in(0,1)$ such that
\[
\abs{ \int \phi \circ T^n \psi d\mu - \int \phi d\mu \int \psi d\mu} \le C \lVert {\phi}\rVert_{BV} \lVert{\psi}\rVert_{\infty} \rho^n
\]
for all $\phi$ of bounded variation and bounded $\psi$.
For a description of the ergodic properties of such maps see~\cite[Chapter 5]{Boyarsky_Gora}.

For this map the expansion rate at a fixed point  $x_0$ is $\lambda=|DT(x_0)|$. The EI for an observable maximized at a fixed point is
$\theta=1-\frac{1}{\lambda}$. At a periodic point of prime period $q$ the corresponding quantities are $\lambda=|DT(x_0)|^q$ and again $\theta=1-\frac{1}{\lambda}$.

\subsubsection{ (B) Hyperbolic toral automorphisms }

We consider hyperbolic toral automorphism of the two-dimensional torus $\T^2$ induced by a matrix 
\[
T = 
\begin{pmatrix}
	a & b \\
	c & d \\
\end{pmatrix}
\]
with integer entries, det(T)= $\pm 1$ and no eigenvalues on the unit circle.
In order to simplify the discussion and proofs, we will assume both eigenvalues are positive.
These maps preserve Haar measure $\mu$ on $\T^2$.
A canonical example  is the Arnold Cap map
\[
\begin{pmatrix}
	2 & 1 \\
	1 & 1 \\
\end{pmatrix}
\]
$\T^2$ will be considered as the unit square with usual identifications with universal cover $\R^2 $.
T preserves the Haar measure $\mu$ on $\T^2$ and has exponential decay of correlations for Lipschitz functions, in the sense that there exists $\rho \in (0,1)$, such that
\[
\abs{ \int \phi \circ T^n \psi d\mu - \int \phi d\mu \int \psi d\mu} \le C \lVert {\phi}\rVert_{Lip} \lVert{\psi}\rVert_{Lip} \rho^n
\]
where C is a constant independent of $\phi, \psi$ and $\lVert.\rVert_{Lip}$ is the Lipschitz norm. For more details we
refer to~\cite[Chapter 3]{Mane}.

For these maps the expansion rate at a fixed point $x_0$ is $\lambda=\lambda_u$, the eigenvalue associated to the unstable direction. The EI for an 
observable maximized at a fixed point is $1-\frac{1}{\lambda_u}$.   At a periodic point of prime period $q$ the corresponding quantities are $\lambda=\lambda_u^q$ and  as before
$\theta=1-\frac{1}{\lambda}$. 

\subsubsection{(C) Coupled systems of uniformly expanding maps. }

Next, we consider a simple class of coupled mixing expanding maps of $[0,1]$, similar to those examined in ~\cite{sandro_coupled}.
 For $\beta>0$ let  $T_{\beta} :[0,1]\to [0,1]$ be defined by $T_{\beta}(x)=\beta x$,(modulo $1$). Such beta-transformations possess an     absolutely continuous invariant probability measure $\mu_{\beta}$
 with density $h_{\beta}$ which is of bounded variation and bounded away from zero.  

 We define an all-to-all coupled system, $F: [0,1]^m \to [0,1]^m$, by
\[
F(x_1, x_2, \dots, x_m) := ( F_1(x_1, x_2, \dots, x_m),\dots,F_m(x_1, x_2, \dots, x_m)),
\] 
where
\begin{equation}
	F_j (x_1, x_2, \dots, x_m) = (1 - \gamma)T_{\beta}(x_j) + \frac{\gamma}{m} \sum_{k=1}^{m} T_{\beta}(x_k),
\end{equation}
for $j \in [1,\dots,m]$. $0<\gamma<1$ is the coupling constant.  All subspaces of the form $x_{i_1} =x_{i_1}=\ldots =x_{i_j}$ are invariant synchrony sets.

For these maps, if $\gamma>0$ is sufficiently small D. Faranda, H. Ghoudi, P. Guirard and  S. Vaienti~\cite{sandro_coupled} have shown:
\begin{enumerate}
	\item there exists a mixing invariant measure $\mu$ with density $\tilde{h}$, bounded above and below away from zero;
	\item the system has  exponential mixing for Lipschitz functions versus $L^{\infty}$ functions.
\end{enumerate}

	 Subspaces of the form $x_{i_1} =x_{i_1}=\ldots =x_{i_j}$, $i_1<i_2 <\ldots < i_j\le m$,  are repelling invariant synchrony sets. For this class of examples the invariant set $S$ 
	 will be taken to be the hyperplane of synchrony $x_1=x_2=\ldots =x_m$ and in this setting  $\lambda=(1-\gamma)\beta$ is the expansion in the 
	 directions orthogonal to $S$ at the point $x\in S$.   The extremal index is given by  $\theta =1- \frac{1}{[(1-\gamma)\beta]^{m-1}}$.  Note that the EI depends upon $m$, the number of cells in the coupled system, and tends to zero as $m$ increases.

\subsection{ Main Results.}

We suppose that $S$ is a subset of $X$, where $(X,\mu)$ is a measure and metric space, and that $S$ has good regularity properties, for example: a point; straight line segment or synchrony subspace.
The main phenomena are illustrated by considering 2  classes of observables, giving the Weibull and Fr\'echet case.
\begin{itemize}
\item[(a) ]  Frech\'et case  $\phi (x)=d(x,S)^{-\alpha}$ ($\alpha >0$);
\item[(b) ] Weibull case  $\phi (x)=C-d(x,S)^{-\alpha}$, where $C>0, \alpha<0$.
\end{itemize}

In applications daily rainfall is typically modeled by a Fr\'echet distribution and daily temperatures by a Weibull distribution. We do not consider the Gumbel case  $\phi (x)=-\ln d(x,S) $ as it does not arise in the applications we consider.
One goal of this paper is to determine to what extent the GEV for a time series of observations $\{\phi (T^j x)\}$ determines the GEV for a derived time series of functionals of the observations $\{ \Phi (\phi (T^j x) )\}$.
The functionals we consider, the minimum and time-averages of a series of $k$ observations, in applications model phenomena such as heat waves and prolonged spells of excessive rainfall.

\section{A scaling lemma.}

We now compare the GEV scaling constants for an observable $\phi$ for the maximal process $M_n=\max \{\phi, \phi\circ T, \ldots, \phi \circ T^{n-1}\}$ with those of the functional $Y=\min \{\phi, \phi\circ T, \ldots, \phi\circ T^{k-1}\} $ or $Y=\frac{1}{k}\sum_{j=0}^{k-1} \phi \circ T^j$
with corresponding  maximal process $B_n=\max_{1\le j\le n-1} Y_j$.

We first consider the scaling in the case of a Fr\'echet distribution, a setting in which the scaling constants $b_n$  in a linear scaling are zero and then the Weibull.
The constant $g(k,T)$ in the lemma below is a constant depending upon $k$ and the map$T$. One of the goals of this paper is to calculate $g(k,T)$ in some
simple examples. 

\begin{lemma}\label{main}

(a) Suppose $\phi$ has a Fr\'echet distribution and 
\[
\mu (M_n \le u_n (\tau) )\to e^{-\theta_1 \tau}
\]
and
\[
\mu (B_n \le w_n (\tau) ) \to e^{-\theta_2 \tau}
\]

If 
\[
w_n (\tau) = g(k,T) u_n (\tau),
\]
then 
\[
\xi_2=\xi_1
\]
\[
\mu_2=g(k,T) \mu_1 \bigl(\frac{\theta_2}{\theta_1}\bigr)^{\xi_1}.
\]
and
\[
\sigma_2=g(k,T) \sigma_1 \bigl(\frac{\theta_2}{\theta_1}\bigr)^{\xi_1}.
\]

(b) Suppose $\phi$ has a Weibull  distribution and

\[
\mu (M_n \le u_n (\tau) )\to e^{-\theta_1 \tau}
\]
\[
\mu (B_n \le w_n (\tau) ) \to e^{-\theta_2 \tau}
\]
Let  $u_n(\tau)=a_n \rho (\tau) + C$,  $w_n(\tau)=\alpha_n \rho (\tau) +C$
and $\alpha_n =g(k,T) a_n$. 
Then 
\[
\xi_2=\xi_1
\]
\[
\mu_2=\frac{\sigma_1}{\xi_1} (g(k,T)-1)+\mu_1 -\frac{g(k,T)\sigma_1}{\xi_1} (1-(\frac{\theta_2}{\theta_1})^{\xi_1})
\]
and
\[
\sigma_2=g(k,T) \sigma_1 \bigl(\frac{\theta_2}{\theta_1}\bigr)^{\xi_1}
\]
where $g(k,T)$ is a constant depending on $k$ and the dynamical system $(T,X,\mu)$.
\end{lemma}

\begin{remark}
In our applications a key calculation is to determine $g(k,T)$ and in particular how this function scales with $k$.
\end{remark}
\begin{remark}
In the second part of the proof below, the Weibull case, it is useful to have an example in mind. Suppose $\phi=C-d(x,x_0)^{\alpha}$ where $x_0\in [0,1]$
equipped with Lebesgue measure $\mu$. It is easy to calculate that the scaling is $\mu( M_n \le C-\frac{t^{\alpha}}{n^{\alpha}}) \to e^{-t}$. Let $\rho=-(t)^{\alpha}$
and $a_n =n^{-\alpha}$ (as $a_n >0$, $\rho<0$  in the standard form for the  Weibull distribution). Then $\mu (M_n \le C +a_n \rho)\to e^{- (-\rho)^{\frac{1}{\alpha}}}$ so that $\alpha=-\xi$.  
We then let $z=a_n\rho+C$ and obtain $\mu (M_n \le z)\to e^{- (-[z-C]/a_n)^{\frac{1}{\alpha}}}$. Finally we compare the expression $e^{- (-[z-C]/a_n)^{\frac{1}{\alpha}}}$
to the standard GEV to estimate the scale, location and shape parameters for fixed $n$.
\end{remark}
\noindent {\bf Proof:}

Assume $\phi$ has a Fr\'echet distribution, then $b_n=0$.  To obtain linear scaling for the sequence $u_n(\tau)$  we write $t=\rho(\tau)=\tau^{- \xi_1}$  where $\xi_1>0$ is the shape parameter for the 
limiting distribution of $M_n$. Thus  $u_n (\tau)=a_n t$ and $w_n (\tau)=g(k,T)a_n\rho(\tau)$. For simplicity of exposition we define $\alpha_n=g(k,T)a_n$.
We obtain
\[
\mu (M_n \le a_n t)\to e^{-\theta_1 t^{-\frac{1}{\xi_1}}}:=G^{\theta_1}(t)
\]
\[
\mu (B_n \le \alpha_n t)\to e^{-\theta_2 t^{-\frac{1}{\xi_1}}}:=G^{\theta_2}(t)
\]

So for fixed large $n$, 
\[
\mu (M_n \le t)\sim G^{\theta_1} (t/a_n)
\]
and
\[
\mu (B_n \le t)\sim G^{\theta_2} (t/\alpha_n)
=\left[G^{\theta_1}\bigl(\frac{t}{a_n}\cdot(\frac{a_n}{\alpha_n}\bigr)\bigr)\right]^{\frac{\theta_2}{\theta_1}}
\]

Note that if $H(t)$ is a generalized EVT distribution with parameters $(\mu, \sigma, \xi)$ then $H(\gamma t)$, $\gamma>0$,  has the same shape parameter $\xi_{\gamma}=\xi$ while $\sigma_{\gamma}=\frac{\sigma}{\gamma}$, $\mu_{\gamma}=\frac{\mu}{\gamma}$. 

Thus if $G^{\theta_1} (t/a_n)$ has parameters $\xi_1$, $\mu_1$ and $\sigma_1$ then $G^{\theta_1}((t/a_n)\cdot(\frac{a_n}{\alpha_n}))$
has GEV parameters $\xi^{'}_2=\xi_1$, $\sigma^{'}_2=g(k,T)\sigma_1$ and $\mu_2^{'}=g(k,T)\mu_1$.

Furthermore it is known~\cite[Theorem 5.2, Page 96]{Coles} that if $G(t)$ has GEV parameters $\xi$, $\mu$ and $\sigma$ then
$G^{\theta}$ has GEV parameters $\xi_{\theta}=\xi$, $\mu_{\theta}=\mu-\frac{\sigma}{\xi}(1-\theta^{\xi})$ and $\sigma_{\theta}=\sigma\theta^{\xi}$.

Thus $[G^{\theta_1}((t/a_n)\cdot(\frac{a_n}{\alpha_n})]^{\frac{\theta_2}{\theta_1}}]$ has GEV parameters $\xi_2=\xi_1$, $\sigma_2=g(k,T) \sigma_1(\frac{\theta_2}{\theta_1})^{\xi_1}$
and $\mu_2=g(k,T)\mu_1-\frac{g(k,T) \sigma_1}{\xi_1}(1-(\theta_2/\theta_1)^{\xi_1})$.

Comparing the two forms of $G^{\theta_1}(\frac{t}{a_n})$, namely $e^{-\theta_1(\frac{t}{a_n})^{-\frac{1}{\xi_1}}}$ and $\exp\bigg\{-\theta_1\bigg[1 + \xi_1\bigg( \frac{t-\mu_1}{\sigma_1}\bigg)\bigg]^{\frac{-1}{\xi_1}}\bigg\}$, we see that the relation $1 = \frac{\xi_1 \mu_1}{\sigma_1}$ holds.
Hence,
\[ 
\mu_2=g(k,T)\mu_1-\frac{g(k,T) \sigma_1}{\xi_1}\bigl(1-\bigl(\frac{\theta_2}{\theta_1}\bigr)^{\xi_1}\bigr)= g(k,T) \mu_1 \bigl(\frac{\theta_2}{\theta_1}\bigr)^{\xi_1}.
\]
This concludes the proof in the case of a Fr\'echet distribution.

Now suppose that $\phi$ has a Weibull distribution,  a setting in which the scaling constants $b_n$  may be taken as $C$, the supremum of the range of $\phi$.
 To obtain a linear scaling for the sequence $u_n(\tau)$  we write $t=\rho(\tau)=-\tau^{- \xi_1}$  where $\xi_1<0$ is the shape parameter for the 
limiting distribution of $M_n$. Thus  $u_n (\tau)=a_n t+C$ and $w_n (\tau)=\alpha_n t +C$ where $\alpha_n =g(k, T) a_n$ by assumption. 
We obtain
\[
\mu (M_n \le a_n t +C)\to e^{-\theta_1 (-t)^{-\frac{1}{\xi_1}}}:=G^{\theta_1}(t)
\]
\[
\mu (B_n \le \alpha_n t + C)\to e^{-\theta_2 (-t)^{-\frac{1}{\xi_1}}}:=G^{\theta_2}(t)
\]

We now put the distribution into the standard form of the GEV. 
So for fixed large $n$, let $z=a_n t+C$, so that 
\[
\mu (M_n \le z)\sim G^{\theta_1} \biggl(\frac{z-C}{a_n}\biggr)
\]
and similarly 
\[
\mu (B_n \le z)\sim G^{\theta_2} \biggl(\frac{z-C}{\alpha_n}\biggr)
=\left[G^{\theta_1}\biggl(\frac{z-C}{\alpha_n}\biggr)\right]^{\frac{\theta_2}{\theta_1}}
.\]
In $G^{\theta_1}$ we have 
\[
1+\frac{\xi_1}{\sigma_1}(z-\mu_1)=\frac{-(z-C)}{a_n}
\]
which gives the relations $1-\frac{\xi_1\mu_1}{\sigma_1}=\frac{C}{a_n}$ and 
$\frac{\xi_1}{\sigma_1}=\frac{-1}{a_n}$.
Similarly if $\alpha_n=g(k,T)a_n$ replaces $a_n$ in $G^{\theta_1}$ we have, as the shape parameter $\xi_1$ does not change, 
$1-\frac{\xi_1\mu_2^{'}}{\sigma_2^{'}}=\frac{C}{g(k,T) a_n}$ and 
$\frac{\xi_1}{\sigma_2^{'}}=\frac{-1}{g(k,T)a_n}$. We obtain $\sigma_2^{'}=g(k,T)\sigma_1$
and $\mu_2^{'}=\frac{\sigma_1}{\xi_1}(g(k,T)-1) +\mu_1$. 

Now we account for the transformation $G^{\theta_1} \to G^{\theta_2}$  by considering the parameters in $[G^{\theta_1}]^{\frac{\theta_2}{\theta_1}}$.
As in the Fr\'echet case we use the fact that if  $G(t)$ has GEV parameters $\xi$, $\mu$ and $\sigma$ then
$G^{\theta}$ has GEV parameters $\xi_{\theta}=\xi$, $\mu_{\theta}=\mu-\frac{\sigma}{\xi}(1-\theta^{\xi})$ and $\sigma_{\theta}=\sigma\theta^{\xi}$.

Hence the parameters for $B_n$ are $\sigma_2=g(k,T) \sigma_1(\frac{\theta_2}{\theta_1})^{\xi_1}$ and $\mu_2=\frac{\sigma_1}{\xi_1} (g(k,T)-1)+\mu_1 -\frac{g(k,T)\sigma_1}{\xi_1} (1-(\frac{\theta_2}{\theta_1})^{\xi_1})$.

\section{$k$-exceedances: $\phi (x)$ maximized at an invariant set.}

We assume that $S$ is a repelling fixed point $x_0$ in the case of (A) or (B)  or repelling invariant hyperplane of synchrony in the case of $(C)$. 
Let 
\[
Y(x)=\min \{\phi(x), \ldots, \phi (T^{k-1} (x))\}
\]
where $\phi (x)=d(x,S)^{-\alpha}$ ($\alpha>0$) in the Fr\'echet case or $\phi (x)=C-d(x,S)^{-\alpha}$ ($\alpha <0$) in the Weibull case .
Recall $M_n (x): =\max \{\phi (x), \phi (Tx), \ldots, \phi (T^{n-1} x)\}$ and $B_n=\max \{Y (x), \ldots, Y(T^{n-1} (x))\}$.

\begin{theorem}\label{theorem_invariant}
In the setting of $(A)$,  $\lambda=|DT(x_0)|$;  in the setting of $(B)$, $\lambda=\lambda_u$, the eigenvalue in the expanding direction;
and in the case of $(C)$, $\lambda=(1-\gamma)\beta$.

(a) Suppose $\phi$ has a Fr\'echet distribution.
If $M_n$ has GEV with parameters $\xi_1$, $\sigma_1$ and $\mu_1$ then $B_n$ has GEV with parameters $\mu_{2}=\lambda^{-(k-1)\alpha}\mu_1$, $\sigma_{2}= \lambda^{-(k-1)\alpha}\sigma_1$ and $\xi_{2}=\xi_1=\xi$.

(b) Suppose $\phi$ has a Weibull distribution.
If $M_n$ has GEV with parameters $\xi_1$, $\sigma_1$ and $\mu_1$ then $B_n$ has GEV with parameters $\mu_{2}=\mu_1$, $\sigma_{2}= \lambda^{-(k-1)\alpha}\sigma_1$ and $\xi_{2}=\xi_1=\xi$.

\end{theorem}

\begin{proof}
We focus on the Fr\'echet case. The Weibull case follows the same argument.
Define $w_n (\tau)$ by $n\mu (Y > w_n(\tau))=\tau$.
It is easy to see that $Y (x)$ is maximized for those points such that $x$ is closest to $x_0$ or $S$ and then $Y (x)=\min \{\phi (x), \phi (Tx), \ldots, \phi (T^{k-1} x) \}=\phi (T^{k-1} x)\sim \lambda^{-(k-1)\alpha} \phi (x)$ by the assumption of uniform repulsion. Thus the set $(Y>w_n(\tau))$ is a scaled version of and has the same shape as that of $(\phi>u_n (\tau))$. Hence the  proofs of $\DD (u_n)$  and $\DD_q^{'} (u_n)$ proceed exactly as in the proofs of Theorem 2.1 and Theorem 2.8 in~\cite{CHN}.

The relation $n\mu (Y_k > w_n(\tau)  )=\tau$ implies $n\mu (\phi (x)  > \lambda^{(k-1)\alpha} w_n(\tau)  )=\tau$ and hence $w_n (\tau)=u_n (\tau) \lambda^{-(k-1)\alpha}$.

The extremal index of $Y$ and the extremal index of $\phi$ are both easily calculated and equal to $\theta=1-\frac{1}{\lambda}$ in the case of $(A)$, $(B)$ and equal to 
$1-\frac{1}{\lambda^{m-1}}$ in the case of $(C)$~\cite{CHN}. 
Thus Lemma~\ref{main} concludes the proof.

\end{proof}

\section{$k$-exceedances: $\phi (x)$ maximized at a generic non-recurrent point.}

We first consider the Fr\'echet case.
Consider  $Y(x) =\min \{ \phi (x), \ldots, \phi (T^{k-1} x) \}$, where $\phi (x)= d(x,x_0)^{-\alpha}$ and $\alpha>0$.
Suppose further that  $x_0$ is non-recurrent. 
Then there exists a $\delta>0$ such that at least one of the iterates $T^j x$, $j=0,\ldots, k-1$ satisfies $d(T^j  x, x_0)>\delta$ and hence $Y(T^n x)$ is uniformly bounded.
Thus the process $Y(T^n(x))$ is in the domain of attraction of a Weibull distribution rather than a 
Fr\'echet distribution.
In general the form of the Weibull distribution cannot be discerned readily from $\phi$, 
except for some special cases. We illustrate with an example.

\begin{example} Consider the doubling map $T(x)=2x\mod 1$, $x\in[0,1]$ local observable
$\phi(x)=d(x,x_0)^{-\alpha}$ and functional $Y(x)=\min\{\phi(x),\phi(Tx)\}.$ Then
the process $B_n(x)=\max_{k\leq n}Y(T^kx)$ follows a Weibull law with tail index of $-1$,  (which is indeed independent of $\alpha$).
To show this, suppose without loss of generality
that $x_0\in(0,1/2)$. By elementary analysis,
it can be shown that $Y(x)$ has global maximum $3^{\alpha}x^{-\alpha}_0$ at value $x=\frac{2x_0}{3}$. Within a local neighbourhood of this maximum, $Y(x)$ is piecewise smooth,
and therefore the level set $\{Y>3^{\alpha}x^{-\alpha}_0-w\}$
has measure asymptotic to $c_{\alpha}w$, for some $c_{\alpha}>0$ and $w\to 0$. Thus,
$Y(x)$ is in the domain of attraction of the Weibull distribution with tail index -1.
Since the level sets for $\{Y>w_n(\tau)\}$ are shrinking intervals,  
the relevant  $\DD (u_n)$  and $\DD_q^{'} (u_n)$ easily hold~\cite{FFT3}. Here, $q$ is chosen according
to the recurrence properties of the orbit of $\frac{2x_0}{3}$ (e.g. periodic versus non-periodic). 
Hence $B_n(x)$ follows a (Weibull) GEV.
\end{example}

We have a similar (non)-result in  the Weibull case. The tail index for $Y$ need not be the same as that of $\phi$ for the same reason as in the Fr\'echet case, though the distribution will also be a Weibull distribution but not that much more can be said in any generality. 

\section{ $k$-averages: $\phi (x)$ maximized at a non-recurrent point. }

In the Weibull case there is little useful that can be said in any generality, except that  $Y$ will also have a Weibull distribution but possibly with a different shape parameter. 
Hence we focus on the  Fr\'echet case. Assume in case $(A)$, $(B)$ or $(C)$  that  $\phi(x)=d(x,x_0)^{-\alpha}$, $\alpha>0$, where $x_0$ is non-recurrent in the sense
of Collet~\cite{Collet}.
In the Fr\'echet or heavy-tailed case the time average is dominated by the maximum value, so we should expect a simple relation for the GEV parameters 
of $M_n=\max_{j\le n} \{\phi (T^j x)\}$ and those of $B_n=\max_{j\le n} \{Y_j\}$ where  $Y_j (x)=\frac{1}{k}\sum_{i=j}^{j+k-1} \phi (T^i x)$.
If $\xi$, $\sigma$ and $\mu$ are the shape, scale and location parameters of $M_n$ then $\xi$, $\frac{\sigma}{k}$ and $\frac{\mu}{k}$ are the shape, scale and location parameters of $B_n$. Assume that  $\phi(x)=d(x,x_0)^{-\alpha}$, $\alpha>0$, where $x_0$ is non-recurrent.

\begin{theorem}\label{thm.avg_generic}
Assume in case $(A)$, $(B)$ or $(C)$  that  $\phi(x)$ is a Fr\'echet observable of form $\phi(x)=d(x,x_0)^{-\alpha}$, $\alpha>0$, where $x_0$ is non-recurrent in the sense
of Collet~\cite{Collet}.
Let  $M_n=\max_{j\le n} \{\phi (T^j x)\}$ and  $B_n=\max_{j\le n} \{Y_j\}$ where  $Y_j (x)=\frac{1}{k}\sum_{i=j}^{j+k-1} \phi (T^i x)$.
If $M_n$ has GEV with parameters $\xi_1$, $\sigma_1$ and $\mu_1$ then $B_n$ has GEV with parameters $\mu_2=\frac{\mu_1}{k}$, 
 $\sigma_{2}=\frac{\sigma_1}{k}$ and $\xi_{2}=\xi_1=\xi$.
\end{theorem}

\begin{proof}
Let $k\ge$ and define  the time-average 
\[
Y_i(x)=\frac{1}{k}\sum_{j=i}^{i+k-1} \phi (T^j x).
\]

Let $u_n$ be an increasing sequence.
Let $r(u_n)$ be defined by $B_{r(u_n)}(x_0)=\{x:\phi(x) >u_n\}$. 

Let $V^{-}_n= \cup_{i=1}^{k-1} T^{-i}B_{r(u_n)} (x_0)$ and $V^0_n= B_{r(u_n)} (x_0)$ and note that because of our genericity condition there exists $K$ such that for all large $n$, 
$\phi|_{{V^{-}_n}}\le K$.

We first note that if $u_n \to \infty$ then $Y_1>u_n$ implies that $T^j (x) \in B_{r(ku_n)} (x_0)$ for one (precisely one) $j^*$ such that  $0\le j \le k-1$.
Note that 
\[\frac{1}{k} \sum_{j\not = j^*,0\le j\le k-1}\phi (T^j x) \le M
\]
and by the existence of a density at $x_0$,
\[ 
|\mu (B_{r(M+k u_n)}(x_0))-\mu (B_{r(k u_n)}(x_0))|=o(\mu (B_{r(k u_n)}(x_0))).
\]
Since $T$ preserves $\mu$,  $\mu (T^{-j} B_{r(u_n)} (x_0))=\mu (B_{r(u_n)} (x_0))$ for $0\le j \le k-1$. 
Hence
\[
\frac{\mu (V^0_{n})}{\mu (V^0_n)+\mu (V_n^{-})}=\frac{1}{k}.
\]

Define
\[
M_n =\max \{ \phi (x), \phi (T x), \ldots, \phi (T^{n-1} (x) )\}.
\]
We define a sequence $u_n(\tau)$ by the  requirement that 
\begin{equation}\label{eq.def-u_n-tau}
    n\mu (\phi >u_n(\tau))=\tau.
\end{equation}

It is well known that conditions $D(u_n)$, $D^{'}(u_n)$ of~\cite{LLR} (see for example~\cite{Extremes_Book}) hold for a generic $x_0$ in our dynamical
settings and hence
\[
\mu (M_n \le u_n (\tau)) \to e^{-\tau}.
\]

We will now relate the scaling constants for $B_n:=\max_{j\le n} \{Y_j (x)\}$
where
\[
Y_i (x)=\frac{1}{k}\sum_{j=i}^{i+k-1} \phi (T^j x)
\]
to the scaling constants for $M_n$.

Consider  a sequence $w_n$ such that
\[
n\mu (Y_k>w_n(\tau))=\tau.
\]
Now as $x_0$ is non-recurrent
\begin{align*} \mu (Y_k >w_n(\tau)) 
    &= \mu\left( \sum_{j=0}^{k-1} \phi (T^j x) >kw_n(\tau)\right)\\
    &= k\mu( \phi > k w_n(\tau) ) + o(1/n).\\
\end{align*}

Since $u_n (\tau)$ is defined by the requirement $n\mu (\phi > u_n (\tau))\to \tau$ it is easy to see that 
$u_n(\tau)=n^{\xi}\tau^{-\xi}$.  We see that we have the
relation
\[
w_n (\tau)=\frac{1}{k}\ u_n \bigl(\frac{\tau}{k}\bigr)
\]
and hence $w_n (\tau)=k^{\xi-1} u_n (\tau)$.

It is clear that $\phi$ has extremal index $1$. 
The scheme of the proof of Condition $\DD_q(u_n)$ is itself somewhat standard~\cite{Dichotomy,FHN} and is a consequence of suitable decay of correlation estimates.
Our genericity assumption on $x_0$ establishes  Condition $\DD'_q(u_n)$ in a standard way.
We will now show that $Y$ has an extremal index $\theta=\frac{1}{k}$.
In our setting
\[
\lim_{n\to \infty} \theta_n=\lim_{n\to \infty} \frac{\mu (A_n^{(q)})}{\mu (U_n)}
\]
exists and equals 
\[
\frac{\mu (V_n^0)}{\mu (V_n^0)+\mu (V_n^{-})}=\frac{1}{k}.
\]

Thus we have $g(k,T)=k^{\xi-1}$ and $\theta_2=\frac{1}{k}$. We conclude from Lemma~\ref{main} that
\[
\sigma_2=k^{\xi-1}(\sigma_1(1/k)^{\xi})=\frac{\sigma_1}{k}
\]
and 
\[
\mu_2=k^{\xi-1}(\mu_1(1/k)^{\xi})=\frac{\mu_1}{k}.
\]






\end{proof}

\section{$k$-averages: $\phi (x)$ maximized at an invariant set.}

Unlike for $k$-averages associated to $\phi$ maximised at a non-recurrent point, there is no
general result covering all of cases (A), (B) and (C). The GEV scaling for $k$-averages depends on the recurrence properties of the system, as well as on the geometry of 
the level sets. For simplicity we consider the doubling map $T(x)=2x\mod 1$ 
on the interval $[0,1]$, and observable  $\phi(x)=d(x,0)^{-\alpha}$, $\alpha>0$. 
This is clearly maximized at the fixed point $0$ of $T$. For $k\geq 1$, consider the time average 
$$Y_i(x)=\frac{1}{k}\sum_{j=i}^{i+k-1}\phi(T^{j}(x)).$$
We state the following result. 

\begin{theorem}\label{thm.double0_avg}
Consider the map $T(x)=2x\mod 1$, and the observable $\phi(x)=d(x,0)^{-\alpha}$.
Let  $M_n=\max_{j\le n} \{\phi (T^j x)\}$ and  $B_n=\max_{j\le n} \{Y_j\}$ where  $Y_j (x)=\frac{1}{k}\sum_{i=j}^{j+k-1} \phi (T^i x)$.
If $M_n$ has GEV with parameters $\xi_1$, $\sigma_1$ and $\mu_1$ then $B_n$ has GEV with parameters $\mu_2=\frac{c(k)\mu_1}{k}$,
 $\sigma_{2}=\frac{c(k)\sigma_1}{k}$ and $\xi_{2}=\xi_1=\xi$.
Here $c(k)$ is a function of $k$ which satisfies $c_l<c(k)<c_u$ for uniform constants $c_l,c_u>0$.
\end{theorem}


\begin{remark}\label{rmk.8.1_1}
This example is clearly within class of systems (A). From the techniques of the proof, it is
possible to extend the conclusion of Theorem \ref{thm.double0_avg} to other examples
within class (A). The main steps are to determine the scaling laws for the GEV constants.
Given this, the verfication of conditions $\DD_q(u_n)$,  $\DD'_q(u_n)$ follows the same
approach as the proof of Theorem~\ref {thm.avg_generic}.
\end{remark}
\begin{remark}\label{rm: index}
Within the proof we will be more explicit about the functions $c(k)$,
and for the extremal index we obtain the asymptotic result
$\theta\sim \frac{1}{2k}$, (for large $k$). 
\end{remark}

\noindent {\bf The case $k=2$.} Before embarking on the proof for general $k$, we illustrate
for $k=2$. Here, we just focus on the calculation of the relevant $w_n$ sequence, and calculation
of the extremal index. In the proof of Theorem \ref{thm.double0_avg} for general $k$, we consider
verification of conditions $\DD_q(u_n)$,  $\DD'_q(u_n)$.

Here, we therefore have 
$$Y_0(x)=\frac{1}{2}\left(d(x,0)^{-\alpha}+d(T(x),0)^{-\alpha}\right).$$
For large $u>0$, consider the set $\{Y_0(x)\geq u\}$. Clearly $\{Y_{0}(x)=+\infty\}=\{0,1/2\}$,
and hence there exists $u_0>0$ such that for all $u\geq u_0$ we have the following.
The set $\{Y_0(x)\geq u\}$ consists of two disjoint neighbourhoods $\mathcal{N}_0(u)$,
$\mathcal{N}_{\frac{1}{2}}(u)$ of $0$, and (resp.) $1/2$. We choose $u_0$ so that
$T^2$ is injective on each neighbourhood. 
Furthermore we suppose that for all $x\in\mathcal{N}_{\frac{1}{2}}(u)$ we
have $d(x,0)>1/4$.

For $u\geq u_0$, we now estimate $\mu\{Y_0>u\}$. Consider $x\in\mathcal{N}_0(u)$
with $d(x,0)=r$. Then
\begin{equation*}
Y_0(x)=\frac{1}{2}( r^{-\alpha}+(2r)^{-\alpha}).
\end{equation*}
By estimating the range of values of $r$ for which $Y_0(x)\geq u$ we obtain,
\begin{equation}
\mu(\mathcal{N}_0(u))=2(2u)^{-\frac{1}{\alpha}}(1+2^{-\alpha})^{\frac{1}{\alpha}}.
\end{equation}
Now let $x\in\mathcal{N}_{\frac{1}{2}}(u)$, and suppose $d(x,1/2)=r$. 
Since $T^2$ is injective on $\mathcal{N}_{\frac{1}{2}}(u)$, we have
$d(T(x),T(1/2))=2d(x,1/2)=2r$. Since $T(1/2)=0$, and using the assumption
$1/4<d(x,0)\leq 1/2$, we obtain the following bounds 
\begin{equation}
\frac{1}{2}(2^{\alpha}+(2r)^{-\alpha})\leq Y_0(x)\leq \frac{1}{2}(4^{\alpha}+(2r)^{-\alpha}).
\end{equation}
Hence
\begin{equation*}
Y_0(x)\geq u\;\implies\; r\leq\frac{1}{2}(2u-4^{\alpha})^{-\frac{1}{\alpha}},
\end{equation*}  
and conversely
\begin{equation*}
r\leq\frac{1}{2}(2u-2^{\alpha})^{-\frac{1}{\alpha}}\;\implies\;Y_0(x)\geq u.
\end{equation*}   
Putting this together leads to
\begin{equation}
\mu(\mathcal{N}_{\frac{1}{2}}(u))\in[(2u-2^{\alpha})^{-\frac{1}{\alpha}},
(2u-4^{\alpha})^{-\frac{1}{\alpha}}].
\end{equation}
Hence, 
\begin{equation}
\mu\{Y_0(x)\geq u\}=(2u)^{-\frac{1}{\alpha}}\left(2(1+2^{-\alpha})^{\frac{1}{\alpha}} 
+(1-c_0(\alpha)(2u)^{-1})^{-\frac{1}{\alpha}}\right),
\end{equation}
with $2^{\alpha}\leq c_0(\alpha)\leq 4^{\alpha}$. Thus for $u\geq u_0$ large, we obtain
the asymptotic relation,
\begin{equation}\label{eq:Y0-est-k2}
\mu\{Y_0(x)\geq u\}\sim (2u)^{-\frac{1}{\alpha}}\left(1+2(1+2^{-\alpha})^{\frac{1}{\alpha}}\right).
\end{equation}
Using the notations of Lemma \ref{main} we can obtain the scaling relations between $u_n(\tau)$ and 
$w_n(\tau)$ via $w_n(\tau)=g(k,T)u_n(\tau)$. Using, 
$\mu\{\phi(x)\geq u\}\sim 2u^{-\frac{1}{\alpha}}$ we obtain $u_n=(2n/\tau)^{\xi}$, where $\xi=\alpha$ is the shape parameter. 
The corresponding $w_n(\tau)$ can be found using equation \eqref{eq:Y0-est-k2}. For
$k=2$, this leads
to
$$g(k,T)=\frac{1}{2}\left(1+2(1+2^{-\alpha})^{\frac{1}{\alpha}}\right)^{\alpha}.$$

To consider the extremal index, it suffices to estimate the measure of the set
$A(u)=\{Y_0(x)\geq u,Y_0(T(x))\leq u\}$. Again we take $u\geq u_0$.
For $u_0\geq \frac{1}{2}(12)^{\alpha}$ we claim that 
$A(u)\cap \mathcal{N}_{\frac{1}{2}}(u)=\emptyset$. To see this, consider 
$x\in \mathcal{N}_{\frac{1}{2}}(u)$ and $d(x,1/2)=r$. Then
$$Y_0(T(x))-Y_0(x)=(4r)^{-\alpha}-d(x,1/2)^{-\alpha}\geq (4r)^{-\alpha}-4^{\alpha}.$$
This follows by assumption that $d(x,0)>1/4$ and that $T^2$ is injective on
$\mathcal{N}_{\frac{1}{2}}(u)$ so that $d(T^2(x),T^2(1/2))=4d(x,1/2)=4r$, with
$T^2(1/2)=0$. By assumption on $u_0$, it follows that $Y_0(T(x))>Y_0(x)\geq u$.
Hence $A(u)=A(u)\cap\mathcal{N}_{0}(u)$.

To find $\mu(A(u))$ we proceed as follows. On $\mathcal{N}_{0}(u)$,
with $d(x,0)=r$ we have 
$$Y_0(T(x))=\frac{1}{2}((2r)^{-\alpha}+(4r)^{-\alpha}),$$
so that $\{Y_0(T(x))\leq u\}$ is the set of points $x$ satisfying
$$d(x,0)\leq \frac{1}{2}(2u)^{-\frac{1}{\alpha}}(1+2^{-\alpha})^{\frac{1}{\alpha}}.$$
Hence, for the extremal index it suffices to study the limit
$$\lim_{n\to\infty}\frac{\mu(A^{(1)}_n(w_n))}{\mu\{Y_0 \geq w_n\}},$$
with $w_n$ specified via the asymptotic relation $n\mu(Y_0>w_n)\to\tau$, (for $\tau>0$).
We obtain,
\begin{equation}
\lim_{n\to\infty}\frac{\mu(A^{(1)}_n(w_n))}{\mu\{Y_0 \geq w_n\}}=\frac{(1+2^{-\alpha})^{\frac{1}{\alpha}}}{1+2(1+2^{-\alpha})^{\frac{1}{\alpha}}}.
\end{equation}
This gives the extremal index, and concludes the example in the case $k=2$.

\begin{proof} We now prove Theorem \ref{thm.double0_avg}.
Following on from the $k=2$ example, we obtain asymptotics for $\mu\{Y_0(x)\geq u\}$,
and $\mu\{Y_0(x)\geq u,Y_0(T(x))\leq u\}$ as $u\to\infty$. The argument
here is for the run-length $k$ arbitrary, but fixed. It is clear that
the set $\{Y_0=+\infty\}$ corresponds to the set $\{x:f^{k-1}(x)=0\}$. We subdivide this set up
into subsets $E_j$ defined as follows:
\begin{equation}
\{Y_0=+\infty\}=\bigcup_{0\leq j< k}\{x:\,f^{\ell}(x)\neq 0,\,\forall\,\ell<j,\,f^j(x)=0\}
:=\bigcup_{0\leq j< k}E_j.
\end{equation}
Thus for $k\geq 4$, we have
$E_0=\{0\}$, $E_1=\{\frac12\}$, $E_2=\{\frac14,\frac34\}$, $E_3=\{\frac18,\frac38,\frac58,
\frac78\}$, etc. The general representation can be expressed as
\begin{equation}
E_j=\left\{\frac{2i+1}{2^j},\,0\leq i\leq 2^{j-1}\right\},\;(1\leq j\leq k-1),
\end{equation}
and we write $x_{i,j}=2^{-j}(2i+1)$. Note that the cardinality of $E_j$ is $2^{j-1}$. To estimate
$\mu\{Y_0>u\}$, we consider the function $Y_0$ on neighbourhoods of $x_{i,j}$. We state
the following lemma.
\begin{lemma}\label{lemma:Y0-est}
For $k\geq 2$, we have the following asymptotic
\begin{equation}\label{eq:Y0-lemma-est}
\mu\{Y_0(x)\geq u\}\sim (ku)^{-\frac{1}{\alpha}}\left\{
2\left(\frac{1-z^{k}}{1-z}\right)^{\frac{1}{\alpha}}+
\sum_{j=1}^{k-1}\left(\frac{1-z^{j}}{1-z}\right)^{\frac{1}{\alpha}},
\right\}
\end{equation}
where $z=2^{-\alpha}$. In particular, the following refinements hold. 
There exists $c_1\equiv c_1(k)>0$ with
\begin{equation}\label{eq:Y0-lemma-est2}
\mu\{Y_0(x)\geq u\}\sim c_1(k)(ku)^{-\frac{1}{\alpha}},\;\textrm{with}\;
1\leq\frac{c_1(k)}{k}\leq\frac{2}{(1-2^{-\alpha})^{-\frac{1}{\alpha}}}. 
\end{equation}
Moreover, as $k\to\infty$ then $c_1(k)/k$ approaches
$(1-2^{-\alpha})^{-\frac{1}{\alpha}}.$ 
\end{lemma}
\begin{proof}
We prove Lemma \ref{lemma:Y0-est} as follows. Following the methods for the case $k=2$, 
we see that there exists $u_0>0$, such that for all $u\geq u_0$ the following properties hold.
\begin{itemize}
	\item[(i)] The set $\{Y_0>u\}$ consists of a union of $2^{k-1}$ disjoint neighbourhoods 
	$\mathcal{N}_{i,j}(u)$ with
	$x_{i,j}\in\mathcal{N}_{i,j}(u)$.
	\item[(ii)] The map $T^k:\mathcal{N}_{i,j}(u)\to S^1$ is injective. Hence,
	$|T^{\ell}(\mathcal{N}_{i,j}(u))|=2^{\ell}|\mathcal{N}_{i,j}(u)|$ for all $\ell\leq k$.
	\item[(iii)] For all $x\in\mathcal{N}_{i,j}(u)$ we have that $d(T^{\ell}(x),0)>2^{-2k}$
	for all $\ell<j$.
\end{itemize}
Thus from (iii), if $x\in\mathcal{N}_{i,j}(u)$ we have that $2^{-2k}\leq d(T^{\ell}(x),0)\leq 1/2$
for all $\ell<j$. The choice $2^{-2k}$ is somewhat arbitrary, and any constant less that $2^{-k+1}$
would suffice. Furthermore if $d(x,x_{i,j})=r$, then
combining with (ii) we see that $d(T^{\ell}(x),0)=2^{\ell}d(x,x_{i,j})=2^{\ell}r$ for all 
$\ell\geq j$. This follows from $T^{\ell}(x_{i,j})=0$ for all $\ell\geq j$.

For $j\geq 1$, let us consider $x\in\mathcal{N}_{i,j}$, and estimate $\mu(\mathcal{N}_{i,j}(u))$ for $u\geq u_0$. Suppose that $d(x,x_{i,j})=r$. We obtain bounds on $r$ in terms of $u$ by considering $\{x:Y_0(x)\geq u\}$.
We have
\begin{equation}
\begin{split}
Y_0(x) &=\frac{1}{k}\left(\sum_{\ell=0}^{j-1}d(T^{\ell}(x),0)^{-\alpha}
+\sum_{\ell=j}^{k-1}d(T^{\ell}(x),0)^{-\alpha}\right),\\
&=\frac{1}{k}\left(\sum_{\ell=0}^{j-1}d(T^{\ell}(x),0)^{-\alpha}
+\sum_{\ell=j}^{k-1} (2^{\ell}r)^{-\alpha}\right).
\end{split}
\end{equation}
From (iii), we have $2^{-2k}\leq d(T^{\ell}(x),0)\leq 1/2$. Writing $z=2^{-\alpha}$, we have the following implications,
\begin{equation*}
Y_0(x)\geq u\;\implies\; r\leq (ku-j z^{-2k})^{-\frac{1}{\alpha}}
\left(\sum_{\ell=j}^{k-1}z^{\ell}\right)^{\frac{1}{\alpha}},
\end{equation*}  
and conversely
\begin{equation*}
r\leq (2u-jz^{-1}) ^{-\frac{1}{\alpha}}\left(\sum_{\ell=j}^{k-1}z^{\ell}
\right)^{\frac{1}{\alpha}}\;\implies\;Y_0(x)\geq u.
\end{equation*}   
In the two implications given above, both bounds on $r$ are asymptotically comparable.
This gives
\begin{equation}
\mu(\mathcal{N}_{i,j}(u))= 2(ku)^{-\frac{1}{\alpha}}\left(1-\frac{h(u,\alpha)j}{ku}
\right)^{-\frac{1}{\alpha}} 2^{-j}\left(\frac{1-z^{k-j}}{1-z}\right)^{\frac{1}{\alpha}}.
\end{equation}
The function $h(\alpha,u)$ satisfies uniform bounds $2^{\alpha} \leq h(u,\alpha)\leq 2^{2\alpha k}$.
Notice that the estimate $\mu(\mathcal{N}_{i,j}(u))$ does not depend on $i$. The dependence on $i$ is removed via item (iii).
Softening of (iii) to include dependence on $i$ would not lead to improvement in the asymptotic
estimates,
except through bounds on the function $h(\alpha,u)$. Now for each $j\geq 1$, there
are $2^{j-1}$ such $\mathcal{N}_{i,j}$. Hence, taking union over all $i$ (for given $j$), 
and by using pairwise disjointness of these neighbourhoods,
we obtain
\begin{equation}\label{eq:est-Y0-1}
\sum_{i=0}^{2^{j-1}-1}
\mu(\mathcal{N}_{i,j}(u))\sim (ku)^{-\frac{1}{\alpha}}\left(\frac{1-z^{k-j}}{1-z}
\right)^{\frac{1}{\alpha}}.
\end{equation}
This equation gives the estimate for $\mu(\{Y_0>u\}\cap(\cup_{i}\mathcal{N}_{i,j}(u)))$.

For case $j=0$, we consider the neighbourhood $\mathcal{N}_{0,0}(u)$ of $x_{0,0}=0$. This estimate
is straightforward, and we obtain the precise result
\begin{equation}\label{eq:est-Y0-2}
\mu(\mathcal{N}_{0,0}(u))=2(ku)^{-\frac{1}{\alpha}}\left(\frac{1-z^{k}}{1-z}
\right)^{\frac{1}{\alpha}}.
\end{equation} 
Using equations \eqref{eq:est-Y0-1}, \eqref{eq:est-Y0-2} we obtain
equation \eqref{eq:Y0-lemma-est} in Lemma \ref{lemma:Y0-est}.
To complete the proof of the lemma, we just estimate the geometric series terms
in equation \eqref{eq:Y0-lemma-est} to quantify the constant $c_1(k)$. 
This analysis is elementary. Indeed,
Since $z=2^{-\alpha}$, we have
$$1\leq \left(\frac{1-z^{k-j}}{1-z}\right)^{\frac{1}{\alpha}}
\leq \frac{1}{(1-z)^{\frac{1}{\alpha}}}=\frac{2}{(2^{\alpha}-1)^{\frac{1}{\alpha}}}.$$
Thus each term within the sum of equation \eqref{eq:est-Y0-1} is bounded above/below accordingly.

To complete the proof the lemma, consider now the asymptotic behaviour as $k\to\infty$, 
so we can refine the estimate on $c_1(k)$.
The following estimates are valid for $|z|<1$, but we have in mind the particular value 
$z=2^{-\alpha}$. We have
$$(1-z^{j})^{\frac{1}{\alpha}}=1-z^{j}b(z),$$
where $b(z)$ uniformly bounded above. Thus the following estimates hold
\begin{equation}
\begin{split}
\sum_{j=1}^{k-1}\left(\frac{1-z^{j}}{1-z}\right)^{\frac{1}{\alpha}} &=\frac{1}{(1-z)^{\frac{1}{\alpha}}}
\sum_{j=1}^{k-1}(1-b(z)z^j)\\
&=\frac{k-1}{(1-z)^{\frac{1}{\alpha}}}+O\left(\frac{z(1-z^{k-1})}{1-z}\right).
\end{split}
\end{equation}
Thus the constant $c_1(k)$ in equation \eqref{eq:Y0-lemma-est2} can be replaced by 
$(1-z)^{-\frac{1}{\alpha}}k+c_2$,
with $|c_2|$ uniformly bounded and independent of $k$. This completes the proof of the Lemma.
\end{proof}


We now consider the extremal index, and it suffices to estimate $\mu\{Y_0(x)>u,Y_0(Tx)<u\}$
for large thresholds $u>0$.
We state the following lemma.
\begin{lemma}
Let $A(u)=\{Y\geq u,Y\circ T<u\}$. Then we have the following estimate
\begin{equation}\label{eq:ei-double-lemma}
\lim_{u\to\infty}\frac{\mu\{A(u)\}}{\mu\{Y_0>u\}}=\frac{1}{2c_1(k)}\left(\frac{1-z^k}{1-z}\right)^{\frac{1}{\alpha}},
\end{equation}
where $c_1(k)$ is the constant in Lemma \ref{lemma:Y0-est}.
\end{lemma}
\begin{proof}
Suppose $u_0$ is as specified in the proof of Lemma \ref{lemma:Y0-est}. 
Consider $x\in\mathcal{N}_{i,j}(u)$ for $u\geq u_0$. For $j\geq 1$, we claim
that $\mathcal{N}_{i,j}(u)\cap\{Y_0(Tx)<u\}=\emptyset$. The claim can be established
by analysing $Y_0(x)$ and $Y_0(Tx)$ for $x\in\mathcal{N}_{i,j}(u)$. Suppose
that $d(x,x_{i,j})=r$. Since $T^k:\mathcal{N}_{i,j}(u)\to S^1$ is injective, we have
\begin{align*}
Y_0(x) &=\frac{1}{k}\left(\sum_{\ell=0}^{j-1}d(T^{\ell}(x),0)^{-\alpha}
+\sum_{\ell=j}^{k-1}(2^{\ell} r)^{-\alpha}\right);\\
Y_0(Tx) &=\frac{1}{k}\left(\sum_{\ell=1}^{j-1}d(T^{\ell}(x),0)^{-\alpha} 
+\sum_{\ell=j}^{k}(2^{\ell}r)^{-\alpha}\right).
\end{align*}
Hence,
\begin{equation}\label{eq:Y0-est-ei}
Y_0(Tx)-Y_0(x)=(2^{k}r)^{-\alpha}-2^{2k\alpha}.
\end{equation}
By adjusting $u_0$ accordingly, the quantity on the right of equation \eqref{eq:Y0-est-ei}
is positive for all $x\in\mathcal{N}_{i,j}(u)$. Indeed, we just require $r<2^{3k}.$
This leads to the claim that $\mathcal{N}_{i,j}(u)\cap\{Y_0(Tx)<u\}=\emptyset$.
For $j=0$, the corresponding intersection is not empty and we estimate the measure 
$\mathcal{N}_{0,0}(u)\cap\{Y_0(Tx)<u\}$. The measure of $\mathcal{N}_{0,0}$ is
given in equation \eqref{eq:est-Y0-2}. Now consider points in this set
with $Y_0(Tx)<u$. Suppose that $d(x,0)=r$, with $Y_0(x)=u$. We estimate $r$ in terms of $u$. 
Since $T^k$ is injective on $\mathcal{N}_{0,0}(u)$ we obtain
\begin{equation}
\frac{1}{k}\cdot r^{-\frac{1}{\alpha}}(z+\ldots+z^k)=ku\;\implies\; 
r=(ku)^{-\frac{1}{\alpha}}z^{\frac{1}{\alpha}}\left(\frac{1-z^{k}}{1-z}\right)^{\frac{1}{\alpha}}.
\end{equation}
Overall, this leads to
\begin{equation}
\mu\{Y_0(x)>u,Y_0(Tx)<u\}=\frac{1}{2}(ku)^{-\frac{1}{\alpha}}\left(\frac{1-z^k}{1-z}
\right)^{\frac{1}{\alpha}}.
\end{equation}
Hence, we obtain
\begin{equation}
\lim_{u\to\infty}\frac{\mu(A(u))}{\mu\{Y_0 \geq u\}}
=\frac{1}{2c_1(k)}\left(\frac{1-z^k}{1-z}
\right)^{\frac{1}{\alpha}}.
\end{equation}
This completes the proof of the lemma.
\end{proof}
We now complete the proof of Theorem \ref{thm.double0_avg}. First we quantify the function
$g(k,T)$ in the relation $w_n(\tau)=g(k,T)u_n(\tau)$. 
Recall that given $\tau>0$, $u_n(\tau)$ is such that 
$n\mu(\phi(x)>u_n(\tau))\to\tau$. Thus $u_n=(2n/\tau)^{\xi}$, where $\xi=\alpha$ is the shape
parameter. Using Lemma \ref{eq:Y0-lemma-est}, we obtain the corresponding $w_n(\tau)$ satisfying
$n\mu(Y_0>w_n(\tau))\to\tau$.  This gives $w_n=\frac{1}{k}(\tau/c_1(k) n)^{-\xi}$. Hence
$$g(k,T)=k^{-1}\left(\frac{2}{c_1(k)}\right)^{-\xi}\approx k^{\xi-1}.$$
For $k\to\infty$, $c_1(k)/k$ approaches $(1-2^{-\alpha})^{-\frac{1}{\alpha}}$, and hence,
$$g(k,T)\sim (2^{\xi}-1)k^{\xi-1}.$$

The extremal index $\theta$ is calculated via 
$$\theta=\lim_{n\to\infty}\frac{\mu(A^{(1)}_n(w_n))}{\mu\{Y_0>w_n\}}.$$
The value of $\theta$ is precisely the quantity in the right-hand side
of equation \eqref{eq:ei-double-lemma}, and we have
\begin{equation}
\theta=\frac{1}{2c_1(k)}\left(\frac{1-2^{-\alpha k}}{1-2^{-\alpha}}
\right)^{\frac{1}{\alpha}}.
\end{equation}
For $k\to\infty$, the value of $\theta$ approaches $1/2k$. Thus, 
from the notations of Lemma \ref{main} we obtain
$\theta_2/\theta_1\to \frac{1}{k}$. That is, when we compare
the extremal index for the process $B_n$ relative to the extremal index for
the process $M_n$.

To conclude the proof, it suffices to verify conditions
$\DD_q(u_n)$,  $\DD'_q(u_n)$. This follows a standard approach, see \cite{FFT3}. 
Indeed the level sets of $\{Y_0>u\}$ are a finite union of intervals which shrink to a finite union of points as $Y_0(x)$ approaches its maximum. Hence indicator functions of these sets
are of bounded variation type (as necessary to allow verification of these two conditions).
This completes the proof.

\end{proof}

\section{Numerical applications}
\subsection{Introduction to numerics for chaotic dynamical systems.}

In this section common numerical techniques are used to estimate the extremal index and the generalized extreme value parameters for the extremal functionals. We provide a brief introduction on numerical simulation for chaotic dynamical systems; however, for the interested reader, we refer to \cite[Chapter 9]{Extremes_Book} for a nice summary of proven numerical techniques that produce reliable statistical properties in this setting.

Numerical simulations to support Theorem~\ref{theorem_invariant} and Theorem~\ref{thm.avg_generic} are explored in the following sections. In general, every simulation utilizes the following procedure: 
\begin{itemize}
\item[1.] Simulate the trajectories, under iterations of a given dynamical map, beginning with a set of random initial conditions chosen uniformly over $X$.
\item[2.] Take an observable function $\phi(x):X\rightarrow \mathbb{R}$ on each trajectory as a measurable observation in $\mathbb{R}$.
\item[3.] Block the observable trajectories from 2. and take the maximum over each block to calculate the sequence $(M_n)$.
\item[4.] Compute the moving $k$ windowed exceedance or $k$ windowed time-average over each trajectory from 2.
\item[5.] Block the trajectories from 4. with the same block length as 3. to calculate the sequence $(B_n)$.
\item[6.] Perform maximum likelihood estimation of the generalized extreme value distribution parameters $\mu_1$, $\sigma_1$, and $\xi_1$, for the sequence $(M_n)$ and $\mu_2(k)$, $\sigma_2(k)$, and $\xi_2(k)$ for the sequence $(B_n)$.
\item[7.] Repeat 4--7 for each window size $k$.
\end{itemize}

Trapping of a trajectory near the a fixed point occurs in numerical simulation of piecewise $C^2$ uniformly example maps, introduced in 3.2.1 (A), due to finite precision. Adding a small $\varepsilon = O(10^{-8})$ perturbation to each step in the orbit allows our trajectory to continue to evolve under longer iterations of the interval map without affecting the statistical properties of extremes of the trajectory. Detailed studies on the effects of small additive error for these systems in the context of extreme value theory are described in \cite[7.5 and 9.7]{Extremes_Book} and \cite{FMT}, and employed in many numerical investigations of extremes in dynamical systems such as \cite{FV, FMT}, for interval expanding maps, and \cite{CHN, sandro_coupled} for coupled systems of uniformly expanding maps.

\subsection{Numerical applications in the Fr\'{e}chet case.}
\subsubsection{ Numerical applications of Theorem \ref{theorem_invariant} (a): $k$-exceedances, maximized at an invariant set} \emph{(A) Uniformly expanding maps of the interval: $k$-exceedances maximized at repelling fixed point.}
We use the doubling map 
\[
T(x)=2x\mod1
\]
to simulate 500 different trajectories of length $10^4$ beginning with 500 randomly chosen (from a uniform distribution) points and calculate the trajectory of the observable $\phi(x) = d(x,x_0)^{-1}$ where $x_0 = 0$ is the repelling fixed point. Note that with this choice of observable we can expect the sequence $(M_n)$ to follow a Fr\'{e}chet distribution, $\xi_1=1$, from \cite{HNT}. We estimate the generalized extreme value parameters using the procedure described in Section 9.1 with the $k$-exceedance functional. From Theorem~\ref{theorem_invariant}, noting that $\lambda = 2$, we expect that
\[ 
\mu_2(k) = 2^{-(k-1)\cdot 1}\mu_1, \quad \sigma_{2}(k) = 2^{-(k-1)\cdot 1} \sigma_1, \quad \xi_2(k) = \xi_1 = 1
\]
We refer to figure \ref{fig:invfrechet} for an illustration of the numerical agreement we obtain for Theorem \ref{theorem_invariant} (a). We report results for $k = 1$ to $10$, beyond this maximum likelihood estimation of the shape parameter becomes unreliable to compute. This makes computational sense since the $k$-exceedance functional limits right-tail sampling, where sampling is vital to estimate the tail decay, by way of taking the minima over an increasing window length. 

In terms of Lemma~\ref{main}, our results indicate that $g(k,T) = 2^{-(k-1)\cdot 1}$ by noting that $\theta_2 = \theta_1$ holds for any $k$. Finally, although we illustrate our results for shape $\xi_1 = \alpha = 1$, we remark that our numerical results, without alterations to the number of initial conditions or length of trajectories, also show agreement with Theorem \ref{theorem_invariant} (a) for $1\ge\alpha>0$ with $g(k,T) = 2^{-(k-1)\alpha}$.

\begin{figure}
\begin{minipage}{0.7\textwidth}
\includegraphics[width=\textwidth]{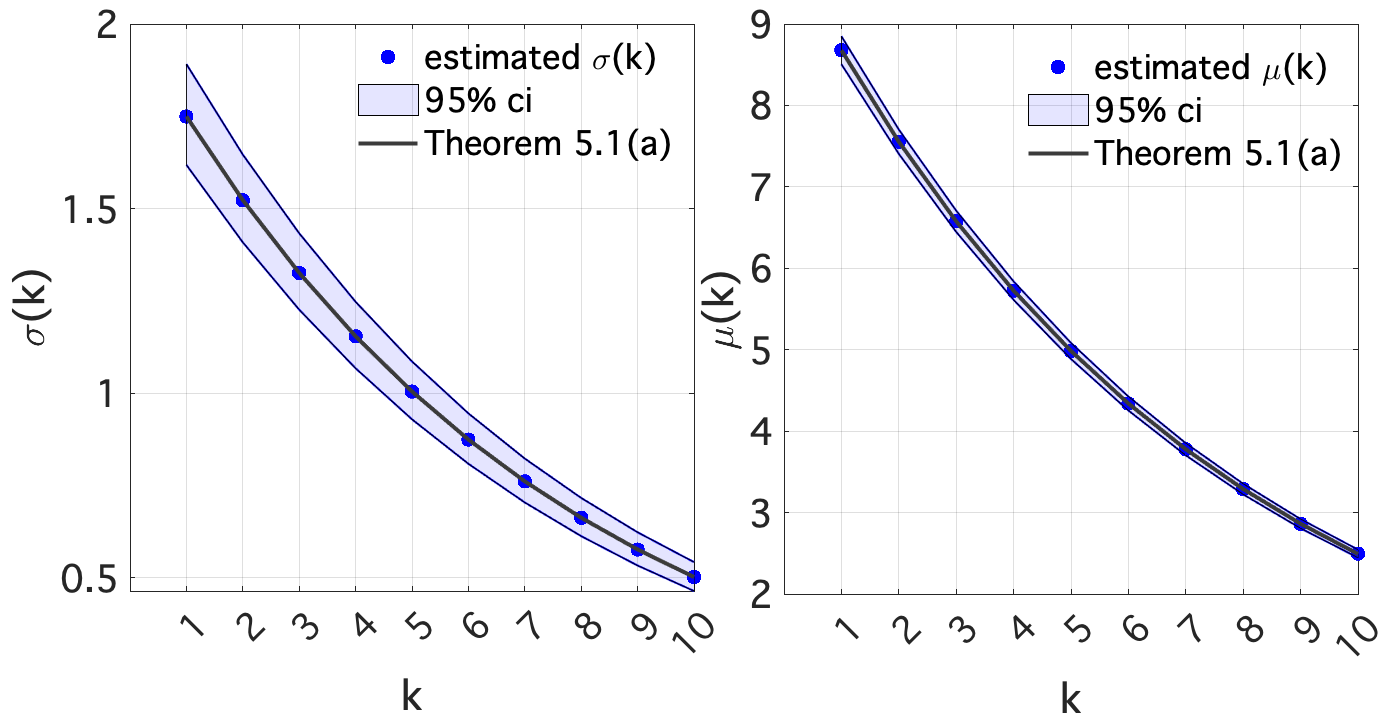}
\caption*{\qquad (a) \qquad \qquad \qquad \qquad \qquad \qquad \qquad (b)}
\end{minipage}
\begin{minipage}{0.7\textwidth}
\includegraphics[width=\textwidth]{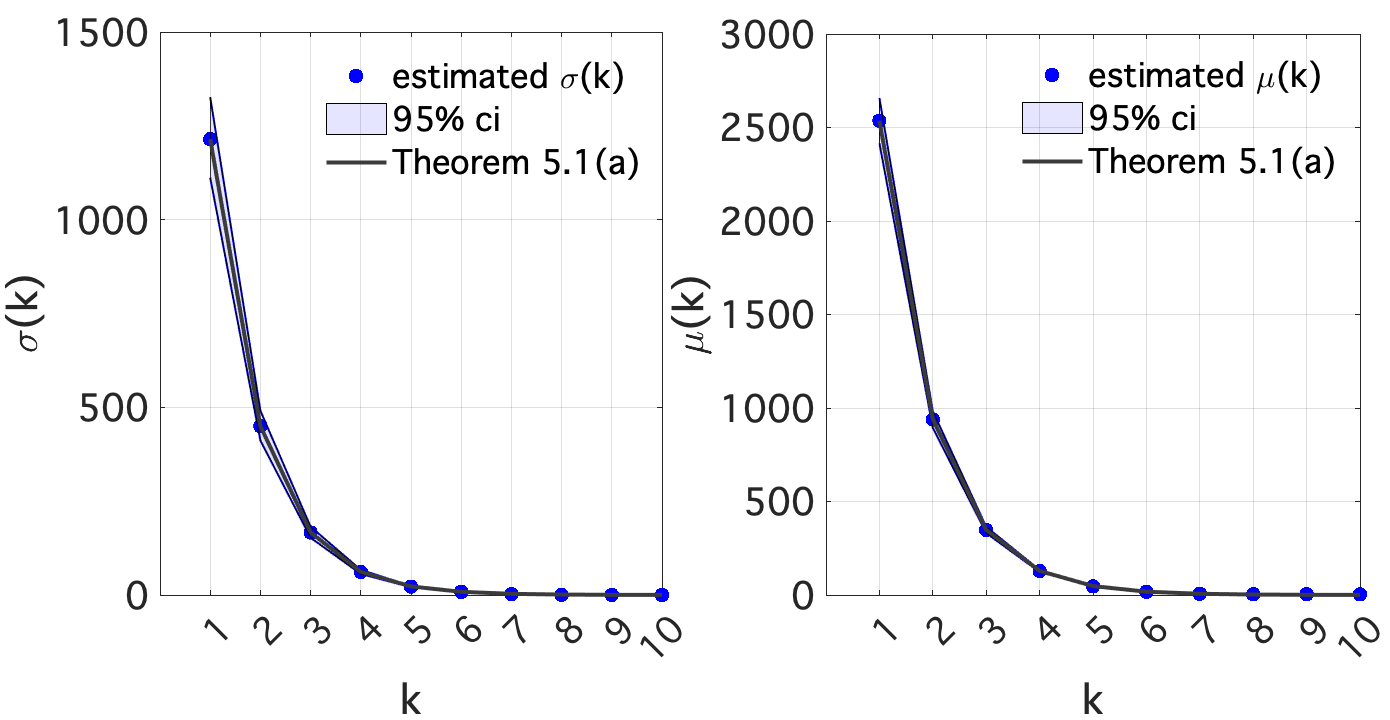}
\caption*{\qquad (c) \qquad \qquad \qquad \qquad \qquad \qquad \qquad (d)}
\end{minipage}
\caption{Numerical agreement with Theorem \ref{theorem_invariant} (a,  b) for the doubling map and (c, d) the coupled system of expanding maps with the $k$-exceedance functional, Fr\'{e}chet observable, and $x_0=0$ as the invariant fixed point.}\label{fig:invfrechet}
\end{figure}

\emph{(C) Coupled systems of uniformly expanding maps: $k$-exceedances maximized at an invariant set.}

We consider the 3-coupled system $F: [0,1]^M \to [0,1]^M$, described in Equation (3.1) of Section 3.2.3, of expanding maps of form $T_{\beta}(x) = \beta x \mod{1}$ for $\beta=3$. In this setting, the uniform expanding factor will be $\lambda=3(1-\gamma)$.

For $x = (x_1,x_2,x_3)$, we define $\|(x_1,x_2, x_3)\|= \sqrt{x_1^2 + x_2^2 +x_3^2} = \|x\|_2$. We take $\phi(x) = d(x,L)^{-1}$ as our observable of this system, where $L$ is the repelling invariant hyperplane of synchrony 
\[
L = \{(x_1, x_2,x_3):  x_1 = x_2 = x_3\}.
\]
Note that the component of a point $x= (x_1,x_2,x_3)$ orthogonal to $L$ is $x^{\perp}=(x_1 - \bar{x},x_2 - \bar{x}, x_3 - \bar{x})$ where $\bar{x}=\frac{1}{3}\sum^3_{i=1}x_i$, so our observable reduces to $\phi(x) = (\|x^{\perp}\|_2)^{-1}$ for $\alpha = 1$. From Theorem~\ref{theorem_invariant}, we expect that,
\[ 
\mu_2(k) = (3(1-\gamma))^{-(k-1)\cdot 1}\mu_1, \quad \sigma_{2}(k) = (3(1-\gamma))^{-(k-1)\cdot 1} \sigma_1, \quad \xi_2(k) = \xi_1 = 1.
\]

We refer to figure \ref{fig:invfrechet} for an illustration of the numerical results for this case which agree well with Theorem \ref{theorem_invariant} (a).

\subsubsection{ Numerical inference using Theorem \ref{theorem_invariant} (a): $k$-exceedances, maximized at an invariant set} 

To fit an extremal distribution to data, one often assumes the form of the underlying distribution and fits a   Generalized Extreme Value (GEV) for block maxima or a Generalized Pareto for over threshold excesses. From this one estimates the parameters (location, scale, and shape), often by some form of likelihood or moments estimation.
Arguably, the most important parameter in extremal fitting is the shape parameter, as it describes the behavior in the tails of the distributions which correspond to the most impactful extreme events.

It is well-known, with research beginning in 1928 with Fisher and Tippett, that estimation of the true shape parameter is the most numerically difficult because of the length of time it takes the extremal distribution to converge in the tails.
Advancements in shape parameter estimation, stemming from the work of Leadbetter (1983) and Smith (1982, 1987), have made reasonable estimations of limiting extremal distributions possible, especially for the bounded, Weibull case and the heavy-tailed, Fréchet case. As an example, the optimal rate of convergence to the asymptotic Fr\'{e}chet extremal distribution was found by Smith (1987) to be $O(n^{\beta})$ for some $\beta$ which depends on the scaling sequence $(a_n)$ coming from normalizing thresholds $u_n = y/a_n+b_n$. As previously mentioned, for any fixed $N$, the standard deviation $\sigma$ plays the role of $(a_N)$ by describing the normalization factor for the maxima $(M_N)$ required based on the rate that $(M_n)\rightarrow\infty$ as $n\rightarrow\infty$. As a result, the theorems and lemmas for the $k$-exceedance and $k$-average functional observables which indicate a reduced $\sigma$ for increasing $k$, illustrate that the value of $\beta$, and hence the convergence rate, is reduced as well.

We now illustrate how to apply Theorem \ref{theorem_invariant} (a), checked for accuracy in the previous section, to obtain accurate estimates of the extreme value distribution for the $k$-exceedance functional and show that these estimates are, for increasing $k$ values, more accurate than maximum likelihood estimation (MLE).
We first explain how we numerically perform fits for Fr\'echet random variables from the doubling map, using the relationships derived, and illustrate that the fits we obtain using our method are as good as the MLE for smaller window sizes $k$ and outperform the MLE for increasing window sizes.
The benefit of working with a simple numerical model initially, is that we can simulate much more data to obtain more reliable empirical estimates of return probabilities.

To apply the results of Theorem \ref{theorem_invariant}, we rely on the Ferro-Segers estimate of the extremal index to estimate the expansion parameter $\lambda$. 
Of course, in theory we know the value of $\lambda$ apriori; however, this value will need to be estimated in the data-based case. 
We estimate the true probabilities and return times of extremes for increasing window size $k$ through empirical estimates from a Monte Carlo method, using a very long simulation of $5\times10^6$ iterations, and take the maxima over $5,000$ blocks of length $10^3$.
We then compare the performance of maximum likelihood estimation and our method, using $500$ blocks of length $10^3$, against this long-run empirical estimate. 
We find that return time plots for window size $k = 1$ using our method is, by definition, the same as the maximum likelihood estimation.
For increasing window sizes $k\ge5$ we see an increasing improvement in return time estimates from our method compared with fits using the MLE, see figure \ref{fig:doubling}(a,b) for illustrations corresponding to window size $k = 5$ and $k = 10$. 
For $k\ge 10$, we observe that the MLE completely fails to estimate a distributional fit, likely due to poor sampling in the tail for increasing $k$ window sizes.

\begin{figure}
\begin{minipage}{0.45\textwidth}
\includegraphics[width=\textwidth]{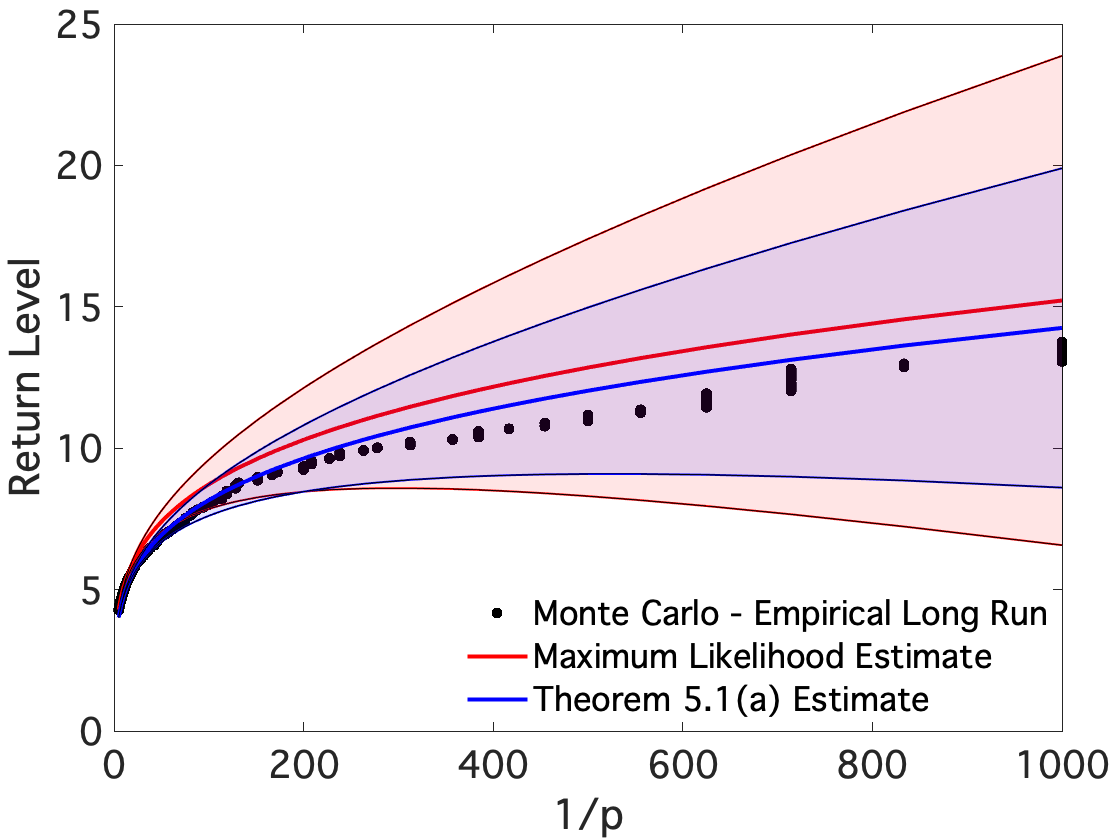}
\caption*{(a)}
\end{minipage}
\begin{minipage}{0.45\textwidth}
\includegraphics[width=\textwidth]{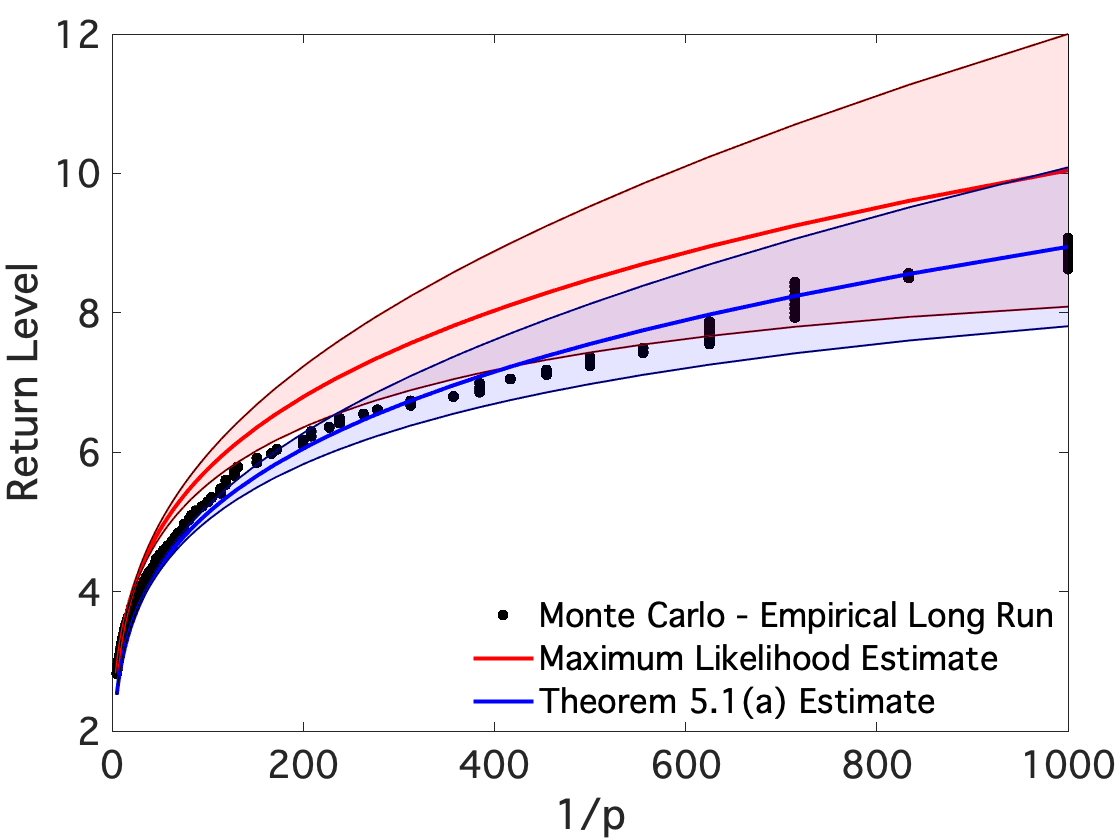}
\caption*{(b)}
\end{minipage}
\raggedright
\begin{minipage}{\textwidth}
\includegraphics[width=0.45\textwidth]{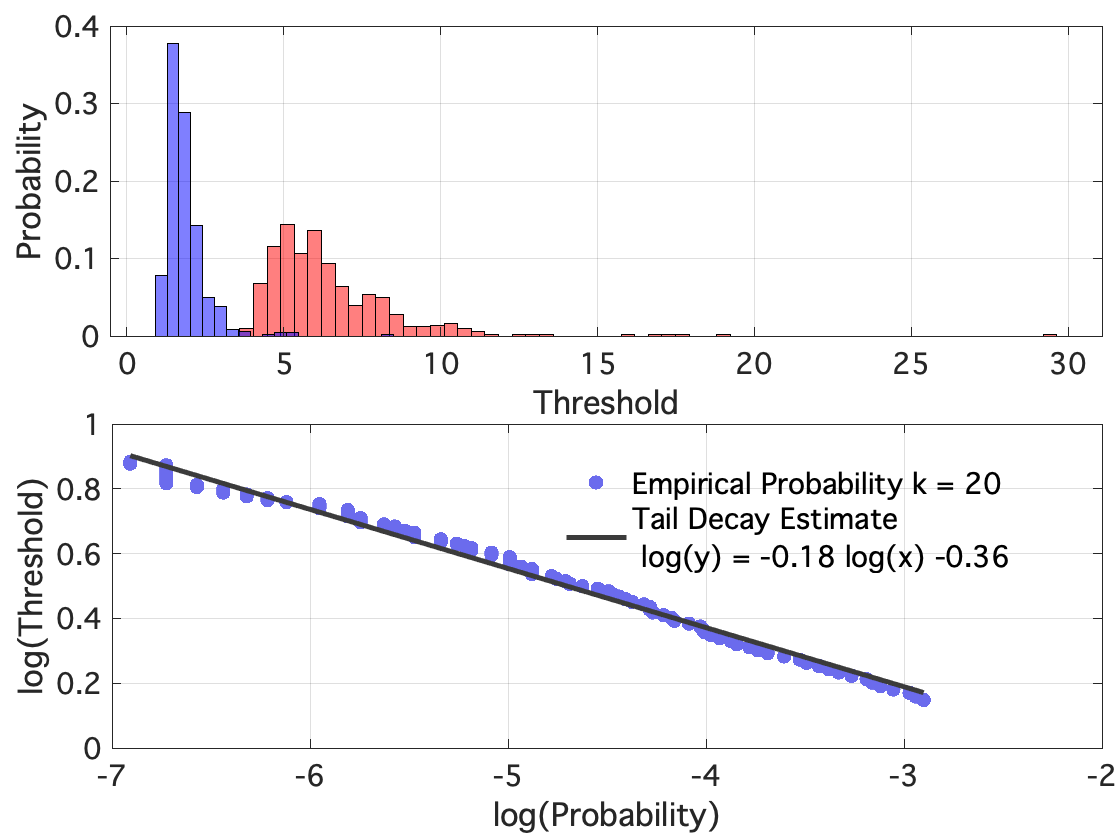}
\caption*{\raggedright \qquad \qquad \qquad (c)}
\end{minipage}
\caption{Doubling Map with $\phi(x) = d(x,0)^{-0.2}$: Return level plots for the block maxima of the $k$-exceedance functional over time windows (a) $k=5$ and (b) $k=8$ applying the results of Theorem \ref{theorem_invariant} (a) verses the maximum likelihood estimated parameters and (c) tail sampling plot and empirical tail probability plot showing tail decay estimate $\propto x^{-0.18}$ corresponding to $k \approx 0.18$.}
\label{fig:doubling}
\end{figure}

Instability of the shape parameter from maximum likelihood estimation as tail sampling becomes rarer contributes to the poor performance that we observe for window sizes $k>10$. 
Sampling issues are illustrated by the distribution in figure \ref{fig:doubling}(c) reflecting the low tail sampling which causes data-based estimates, like the MLE, to become unreliable.
On the other hand, the tail decay remains as expected for the data, even for window sizes as large as $k=20$.
One can check this by fitting a power law to the empirical tail probabilities at high thresholds as seen in figure \ref{fig:doubling}(c).
As a result, we can be confident that the shape parameter we estimated in the original time-series provides a good representation of the windowed $k$-exceedance for window size $k\le 20$.
To numerically check higher values of $k$, we would need to simulate more data for proper estimates of the tail since for $k>20$ we find the tail no longer has the expected power law decay.

\subsubsection{Numerical applications of Theorem \ref{thm.avg_generic}: $k$-average maximized at a non-recurrent point}

\emph{(A) Uniformly expanding maps of the interval: $k$-average maximized at a non-recurrent point.}

We now consider the average functional where $\phi(x) = d(x,x_0)^{-1}$ is maximized at $x_0 = 1/\pi$, a non-recurrent point for the doubling map $T(x) = 2x\mod 1$. Before we state the results, we remark that the choice of $\alpha = 1$ holds numerical significance for this case. The authors in \cite{FV} found that for Gumbel observables defined by $\psi(x) = -\log(d(x,x_0))$, estimates of $b_n$ (for some fixed $N$, this plays the role of the parameter $\mu$), are $x_0$ density-inverse affected by noise, \cite[Prop 7.5.1]{Extremes_Book}. In other words, if $x_0$ is highly recurrent, and therefore has a reasonably large local density\footnote{"local density" here refers to the measure of an $\varepsilon$ ball about $x_0$ on the map $T$, where $\varepsilon$ is on the order of the noise.} then the noise in the system has a lower effect than it does for those points $x_0$ that are non-recurrent (or sporadically recurrent). Choosing a shape parameter $\alpha = 1$ provides enough local density around the point $x_0 = 1/\pi$ to ensure that our estimates for $b_n$ align with the theoretical estimates\footnote{$b_n$ is taken as in equation (9.1.8)}. An alternative to this is to simulate large amounts of data; however, this comes at an obvious additional numerical cost. Interestingly, estimates of $a_n$, and hence estimates of the scale parameter $\sigma$ for some fixed $N$, are unaffected by choice of $x_0$ and will hold for any $1\ge\alpha>0$ readily. Indeed, we see this in our investigation as well. We now state our results.

From Theorem~\ref{thm.avg_generic}, we expect that
\[ 
\mu_2(k) = \frac{\mu_1}{k}, \quad \sigma_{2}(k) = \frac{\sigma_1}{k}, \quad \xi_2(k) = \xi_1 = 1.
\]
We refer to figure \ref{fig:nonfrechet} for an illustration of the numerical results for this case which agree well with Theorem \ref{thm.avg_generic}. In terms of Lemma~\ref{main}, we remark that $\theta_2(k) =\frac{\theta_1}{k}$ as illustrated in figure \ref{fig:nonfrechet}(c) and therefore $g(k,T) = k^{\xi-1}$, as expected.

\begin{figure}
\includegraphics[width=\textwidth]{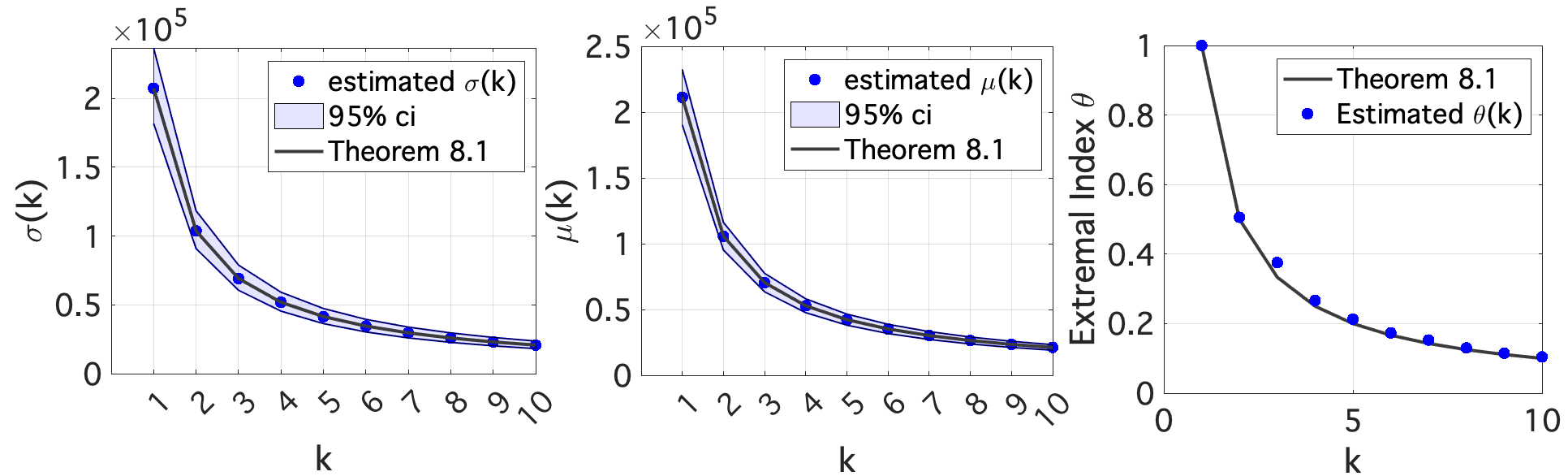}
\caption*{\qquad (a)\qquad \qquad \qquad \qquad \qquad \qquad (b)\qquad \qquad \qquad \qquad \qquad \qquad (c)}
\caption{Numerical agreement with Theorem \ref{thm.avg_generic} for the doubling map and with the $k$-average functional, Fr\'{e}chet observable, and $x_0=1/\pi$ as the non-recurrent point.}
\label{fig:nonfrechet}
\end{figure}

\subsection{Real-world rainfall data (Germany): Climate applications in the Fr\'{e}chet case.}
In this section, we explore the application of Theorem \ref{theorem_invariant} (a) of the $k$-exceedance functional and Lemma \ref{main} (a) of the $k$-average functional for observables with an extremal Fr\'{e}chet distribution. Interestingly, we have found that although real-world rainfall extremes almost always follow a Fr\'{e}chet distribution, general circulation models, like PlaSim give Weibull, or more rarely, Gumbel distributions of rainfall extremes. As a result, our investigation begins with an application of real-world rainfall data taken from the Deutscher Wetterdienst for stations throughout Germany~\cite{dwd}. 

Rainfall from Germany provides the ideal data as it is \textit{mostly} stationary, in the sense that standard cycles have little effect on the total daily rainfall (on the order of $\pm 2$ mm per day).
Occurring at midlatitudes, the quantity of daily rainfall is also not impacted by climate-oceanic cycles, like the El Ni\~{n}o Southern Oscillation.
In many cases, indeed as seen here, we expect daily rainfall extremes to follow a Fréchet distribution.
Results on the $k$-exceedance functional of window size $k$, for example, can tell us how often we expect the minimum daily rainfall will be above some high threshold for all of the $k$ consecutive days.
Quantity of available data on daily rainfall is limited by technology, geography, and time, among other things. 
With the inclusion of a windowed $k$-average or $k$-exceedance functional we also reduce the amount of stochasticity in the data, so measurements of extremes in the tails become rarer. 
This reduction in tail sampling affects our estimates of the parameters of the extreme value distribution.
It is an important question to ask whether we can apply the previous scaling results to real-world Fr\'echet data, for example data representing daily rainfall extremes. 

Scaling lemmas and theorems like the ones in this paper can provide apriori expectations on the extremal distributions of these functionals that allow us to estimate the shape parameter more accurately by applying the relation to estimates using the original time series (e.g. $k = 1$) where the rate of convergence is maximized over all $k$. How? Since measurements of extremes in the tails become rarer by design of the windowed $k$-average or $k$-exceedance functional, provided the shape parameter does not change as a function of the window size, it would be optimal to estimate the shape parameter using the extremes of daily rainfall, keep this as the true shape parameter for the windowed average and minimum daily extremes, and then extrapolate the distribution by using the derived relationship between the location and scale parameter and the window size.

First, we show that the expected behavior of the location and scale parameter, with respect to window size, is verified for the rainfall data in both the $k$-exceedance and $k$-average functional.
Then, we illustrate how to numerically perform fits using our results for the $k$-exceedance functional.
(We do not have enough data to prove empirical fits for the $k$-average functional; however, the results would still hold if enough data were available.)
Finally, we show that (1) our results are as good as the Maximum Likelihood Estimate (MLE) for smaller window sizes and (2) outperform the Maximum Likelihood Estimate for increasing window sizes.
This holds true until the tail can no longer be represented by the time series for large window sizes. 

We now check the numerical results shown above for simple, chaotic dynamical systems against the behavior of extremes for daily rainfall data from stations across Germany.
Daily rainfall is maximized in a storm system; hence we expect the time-series to behave like an observable $\phi(x)$ maximized at an invariant set.
Since we can reasonably assume that our daily rainfall extremes follow a Fréchet distribution, our expectation is that we are in the setting of Lemma \ref{main} and Theorem \ref{theorem_invariant} for the $k$-average and $k$-exceedance functional, respectively.
To show this is true, we approximate the relationship between the maximum likelihood estimate of the generalized extreme value distribution corresponding to yearlyblock maxima.
In total, we have 23 years (1995-2018) of daily rainfall data taken from 60 stations. That is, we have $23\times 365\times 60 =  503,700$ data points (approximately, with a small number of missing values replaced by interpolation). 
We perform maximum likelihood estimation of the yearly block maxima for window sizes $k = 1:10$, for the $k$-average functional and $k = 1:7$ for the $k$-exceedance functional.
Window sizes larger than these values result in low tail sampling and hence, unstable shape estimates which heavily affect the accuracy of the MLE. 

From Lemma \ref{main}, our expectation for the location of the $k$-average functional is that $$\mu_{2}(k) = g(k,T)\mu_1$$ where $\mu_1$ is the location parameter for the yearly block maximum of daily rainfall and $\mu_{2}(k)$ is the location parameter for the yearly block maximum of the $k$-average functional. If we let $y = \log(\mu_{2}(k))$ and $x = \log(k)$,  then if $g(k,T)\sim(O(k^{-\alpha}))$, we expect a linear fit,
\[
y = mx+\log(\mu_1)\tag{$\ast$}
\]
with slope $m$ and intercept $b = \log(\mu_1)$. The generalized linear model that is fit to the data using the relationship described by $(\ast)$ results in $m\approx -0.7$ and $b\approx 3.37 = \log(\mu_1)$, so that $g(k,T)$ is estimated as,
\[
g(k,T) \approx k^{m} = k^{-0.7}
\]
We validate this result on $\sigma_{2}(k)$ by showing that the relation,
\[
\sigma_{2} = g(k,T)\sigma_1 \approx k^{-0.7} \sigma_1
\]
holds true for our estimated $g(k,T)$. See figure \ref{fig:rainfall1}(a,b) for an illustration of the fit results for the $k$-average functional of yearly block maximum.

\begin{figure}
\begin{minipage}{\textwidth}
\includegraphics[width=\textwidth]{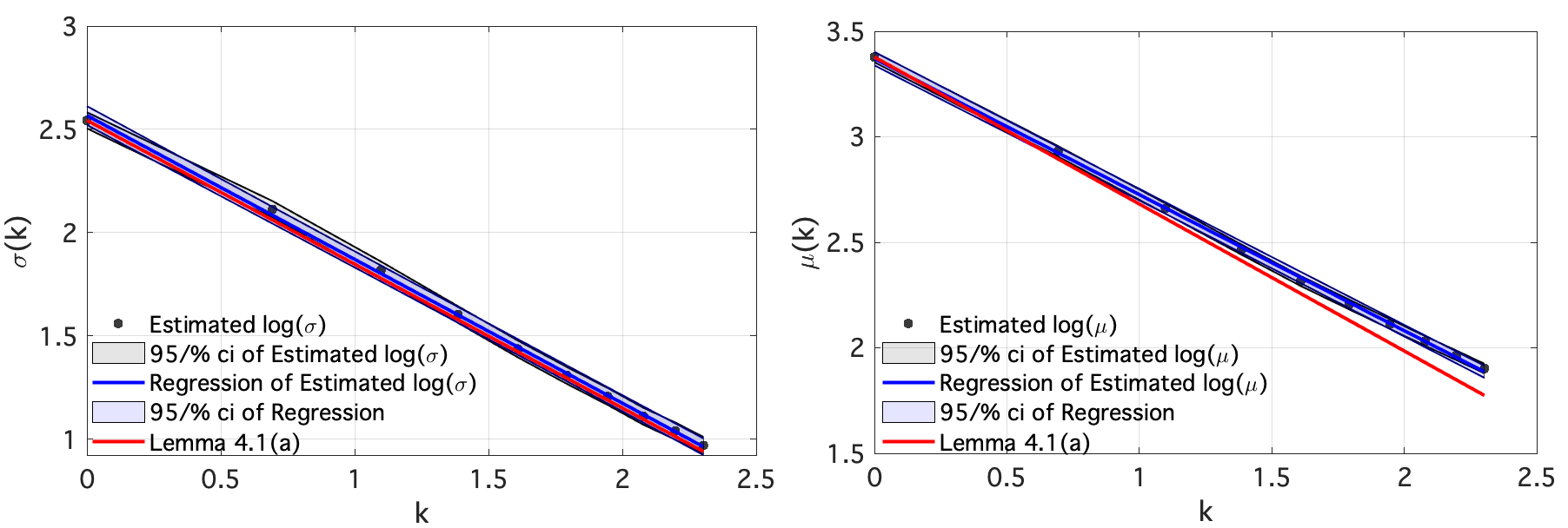}
\caption*{\qquad (a) \qquad \qquad \qquad \qquad \qquad \qquad \qquad \qquad \qquad (b)}
\end{minipage}
\begin{minipage}{\textwidth}
\includegraphics[width=\textwidth]{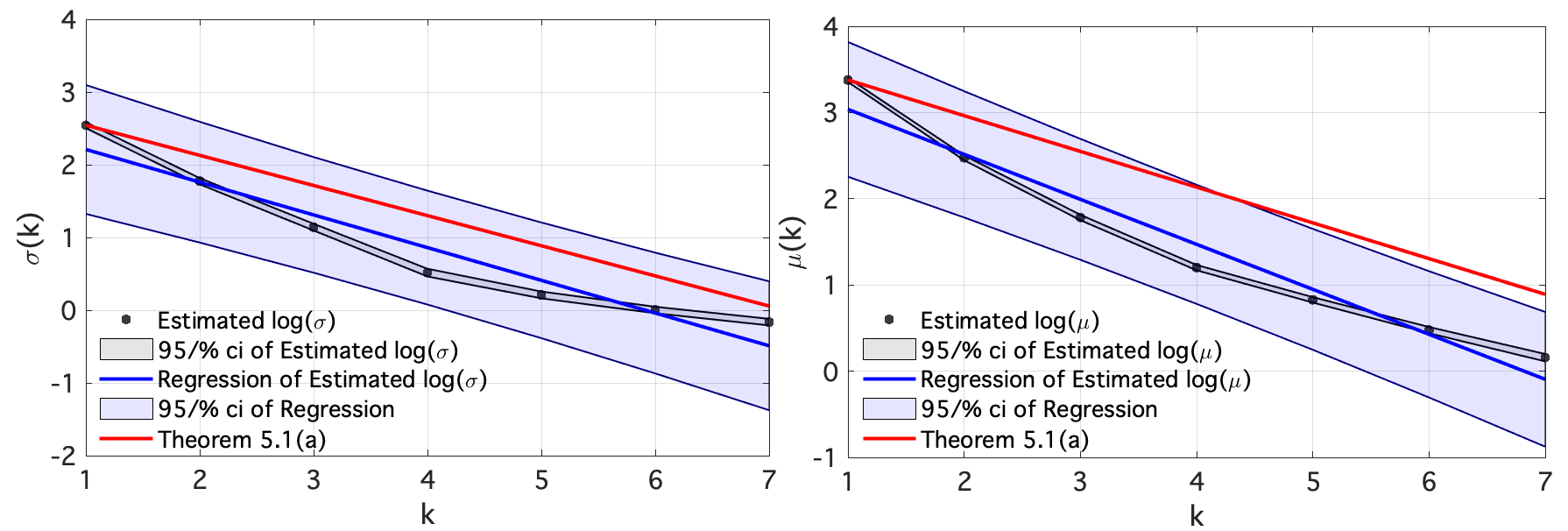}
\caption*{\qquad (c) \qquad \qquad \qquad \qquad \qquad \qquad \qquad \qquad \qquad (d)}
\end{minipage}
\caption{Rainfall in Germany location and scale relation for (a, b) $k$-average functional and (c, d) $k$-exceedance functional showing agreement with Lemma 4.1(a) and Theorem \ref{theorem_invariant} (a), respectively. }
\label{fig:rainfall1}
\end{figure}

From Theorem \ref{theorem_invariant}, our expectation for the location of the $k$-exceedance functional is that $$\mu_{2}(k) = \lambda^{-(k-1)\alpha}\mu_1$$ where $\mu_1$ is the location parameter for the yearly block maximum of daily rainfall, $\mu_{2}(k)$ is the location parameter for the yearly block maximum of the $k$-exceedance functional, and $\lambda$ is the expansion rate parameter.
Letting $y = \log(\mu_{2}(k))$ and $x = k-1$, provided our model is correct, we expect a linear fit,
\[
y = -\alpha\log(\lambda) x + \log(\mu_1)
\]
where $\alpha = \xi$ and hence our slope is $m = -\xi \log(\lambda)$ and our intercept $b = \log(\mu_1)$. 
In practice, we do not know the expansion rate of the map, $\lambda$, because we do not know the underlying complex map (or flow) that completely describes the climate; however, we can estimate such a $\lambda$ by the derived relationship $\theta = 1-\frac{1}{\lambda}$, where $\theta$ is the extremal index.
Hence, if our model is correct, we expect $1-\frac{1}{e^{-m/\xi}} = \theta$ and $b = \log(\mu_1)$.
We estimate the extremal index $\theta$ from the time series by Ferro-Segers and compare the numerical estimation of extremal index we obtain by estimating $\lambda$ from the slope $m$ of the generalized linear model.
See figure \ref{fig:rainfall1}(c,d) for an illustration of our numerical results.
We find our model performs reasonably well with $1-\frac{1}{e^{-m}} \approx 0.35$, ($\theta = 0.34$, approximated by Ferro-Segers) and $b \approx 4.08$ ($\log(\mu_1) = 4.34$, approximated by the MLE). 
Similar results are found for the scale parameter.


\begin{figure}
\begin{minipage}{0.45\textwidth}
\includegraphics[width = \textwidth]{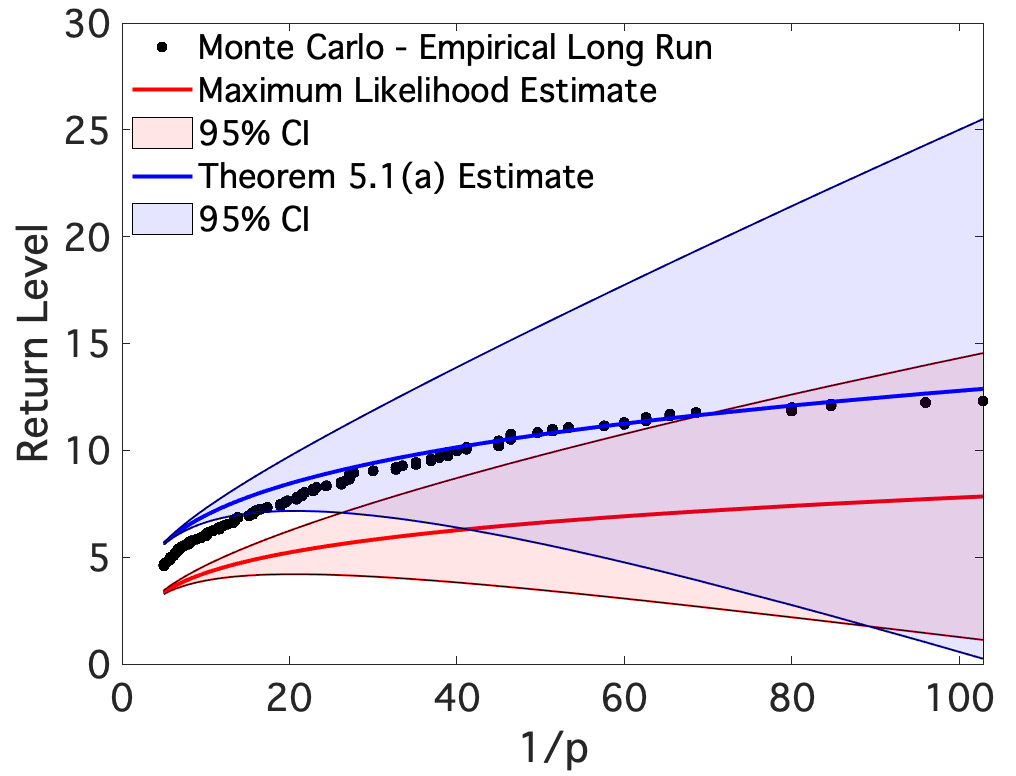}
\caption*{(a)}
\end{minipage}
\begin{minipage}{0.45\textwidth}
\includegraphics[width=\textwidth]{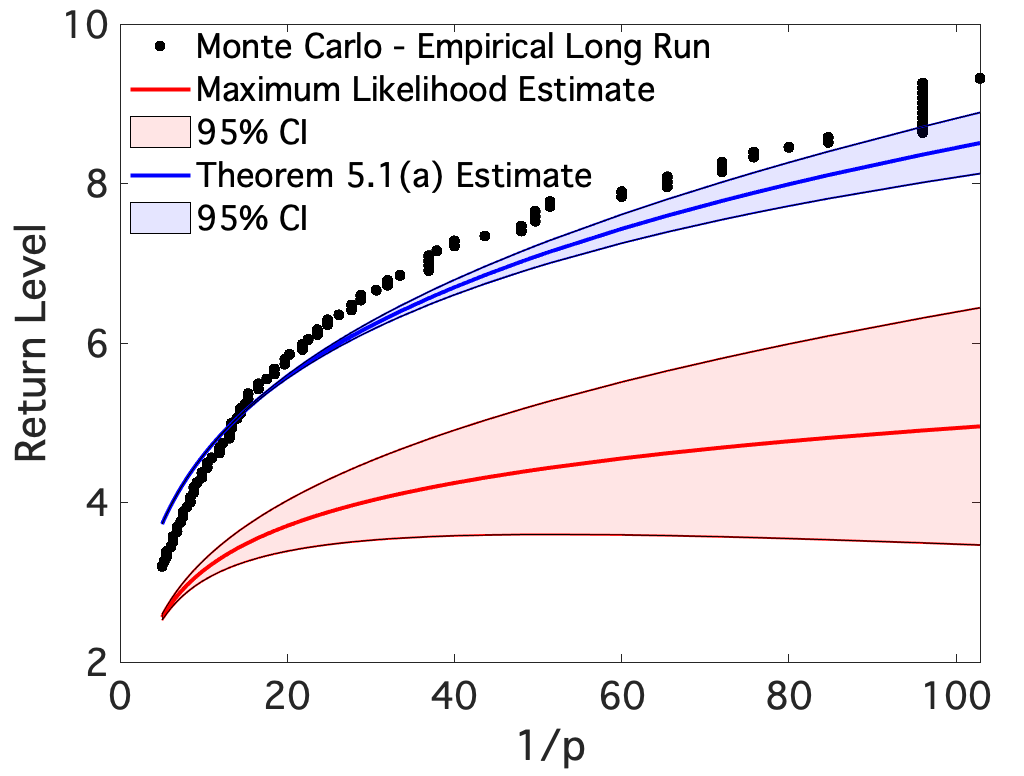}
\caption*{(b)}
\end{minipage}
\label{fig:rainfall3}
\caption{Return level plots using Theorem \ref{theorem_invariant} (a) verses maximum likelihood estimation for (a) $k=5$ and (b) $k=6$. Probabilities for the block maxima of the $k$-exceedance functional for $k=5$ (similarly, $k=6$) correspond to the yearly maximum of 5 (similarly, 6) consecutive daily rainfall values occurring over some sequence of high rainfall thresholds. *95\% confidence intervals do \textit{not} include error from GLM fits for the function $g(k,T)$.}
\end{figure}

We now estimate the return level plots of rainfall data from weather stations across Germany in the same way as Section 9.2. 
Recall that, in total, we have 23 years (1995-2018) of daily rainfall data taken from 60 stations. 
We use all available data to estimate the empirical tail probability estimates for maxima taken over block sizes of 1 year of daily rainfall, in total we have 1,380 block maxima.
We repeat this for increasing window sizes $k = 1:8$, where $k>8$ results in unreliable estimates as there is no longer enough available data to sample the tail.
We then limit our number of block maxima to 100 block maxima, run maximum likelihood estimation to estimate the parameters of the generalized extreme value distribution, and use our method to estimate these parameters.
We compare the resulting tail estimates (80-99.99\% quantiles) to the empirical estimate of the tail probabilities.
For window size $k \ge 3$, we obtain significantly better fits using our results, compared to maximum likelihood estimation.
One clear reason is the fluctuation in the shape parameter that is observed in for the MLE.

\subsection{Numerical Applications in the Weibull case.}
Given that we benchmark the numerical theorems in this paper against maximum likelihood estimates of the parameters for the GEV, we limit this discussion to shape parameters $\xi>-0.5$, where standard regularity properties, such as asymptotic normality, can be assumed on the maximum likelihood estimator \cite{Smith}. We consider the observable $\phi(x) = C-d(x,S)^{-\alpha}$ for $\alpha<0$ which guarantee the shape parameter $\xi>-0.5$. 

\subsubsection{ Numerical applications of Theorem \ref{theorem_invariant} (b): $k$-exceedances, maximized at an invariant set}

Consider the doubling map $T(x) = 2x\mod 1$ equipped with the observable $\phi(x) = 1-d(x,0)^{-0.4}$. From Theorem \ref{theorem_invariant} (b), we expect,
\[ 
\mu_2(k) = \frac{\sigma_1}{\xi_1}(2^{(k-1)\cdot 0.4}-1)+\mu_1, \quad \sigma_{2}(k) = 2^{(k-1)\cdot 0.4} \sigma_1, \quad \xi_2(k) = \xi_1 = 1
\]
We refer to figure \ref{fig:doublingweibull} for an illustration of the numerical agreement we obtain for Theorem \ref{theorem_invariant} (b). In terms of Lemma 4.1 (b), we remark that $\theta_2(k) = \theta_1$ for all $k$ in this example, as expected.

\begin{figure}
\includegraphics[width=0.8\textwidth]{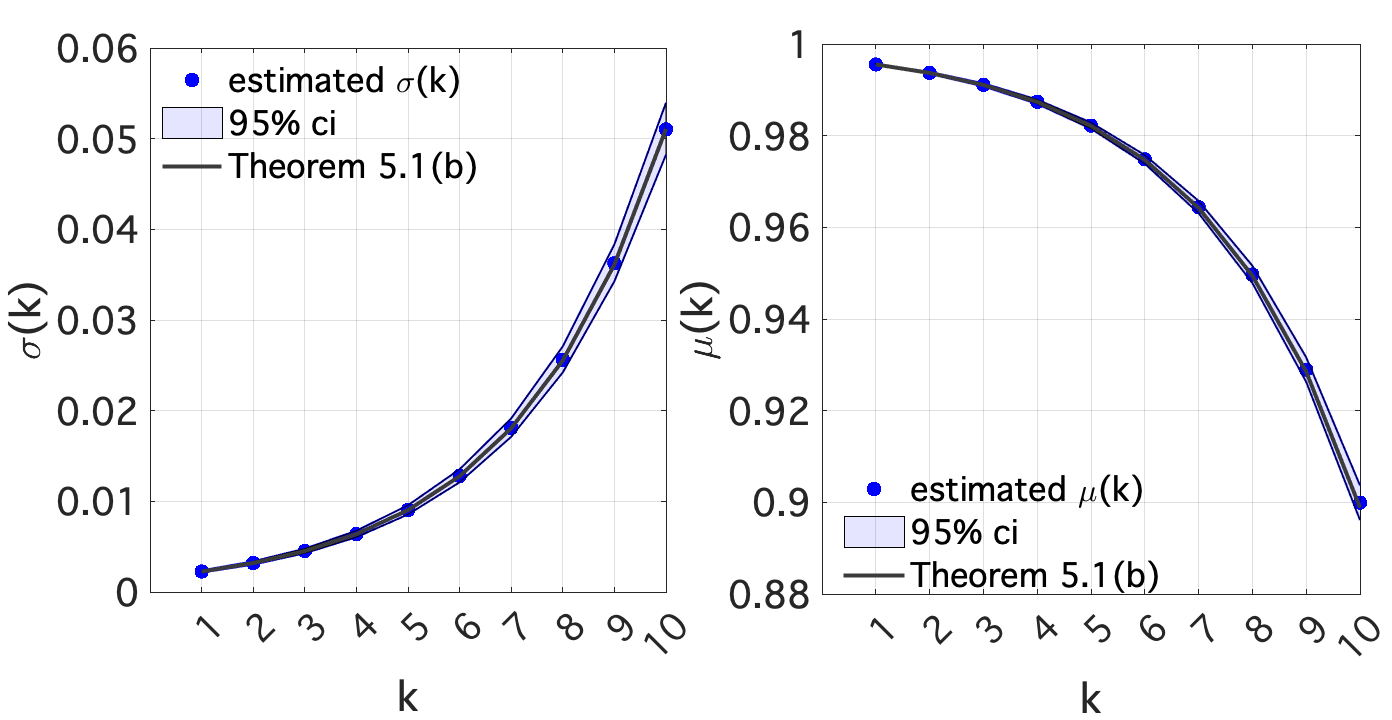}
\caption*{\qquad (a) \qquad \qquad \qquad \qquad \qquad \qquad \qquad (b)}
\caption{Numerical agreement with Theorem \ref{thm.avg_generic} for the doubling map and with the $k$-average functional, Weibull observable, and $x_0=0$ as the invariant fixed point.}
\label{fig:doublingweibull}
\end{figure}

\subsection{Simulated temperature data (PlaSim): Climate applications in the Weibull case.}
Planet simulator (PlaSim) is a general circulation model (GCM) of intermediate complexity developed by the Universit\"{a}t Hamburg Meteorological Institute \cite{plasim}. Like most atmospheric models, PlaSim is a simplified model derived from the Navier Stokes equation in a rotating frame of reference. A stripped version of the PlaSim model, known as PUMA, is a simple GCM with all the processes of PlaSim excluding moist processes. The model structure is described by five partial differential equations which allow for the conservation of mass, momentum, and energy.

The system of five differential equations are solved numerically with discretization given by a (variable) horizontal Gaussian grid and a vertical grid of equally spaced levels so that each grid-point has a corresponding latitude, longitude and depth triplet. (The default resolution is 32 latitude grid points, 64 longitude grid points and 5 levels.) At every fixed time step t and each grid point, the atmospheric flow is determined by solving the set of model equations through the spectral transform method which results in a set of time series describing the system; including temperature, pressure, zonal, meridional and horizontal wind velocity, among others. The resulting time series can be converted through the PlaSim interface into a readily accessible data file (such as netcdf) where further analysis can be performed using a variety of platforms. We refer to \cite{plasim} for more information.

For our purposes, we only consider the extremes of the output of temperature simulated by PUMA with daily and yearly cycles removed before simulation at a single, randomly selected latitude and longitude coordinate pair and a ground-level depth.

We first consider the $k$-exceedance functional so that from Theorem \ref{theorem_invariant} (b), we expect,
\[
\sigma_2(k) = g(k,T)\sigma_1,
\]
where $g(k,T) = \lambda^{-(k-1)\xi}$. We cannot assume $\lambda$ is exactly an expansion parameter for the system; however, we do not need this assumption for estimation. In fact, following the strategy from previous sections, we fit a GLM to the log-transformed data given by,
\[
\log(\sigma_2(k)) = -(k-1)\xi\lambda+\log(\sigma_1).
\]
By letting $y = \log(\sigma_2(k))$ and $x = k-1$, we can expect a linear fit of the data with slope $m = \xi\lambda$ and intercept $b = \log(\sigma_1)$. Our numerical results indicate that (1) a linear fit is indeed a good approximation to this log-transformed data (e.g. the model structure for $\sigma_2(k)$ follows what is expected) and (2) the intercept is within a small, $O(10^{-3})$ or less, tolerance of the expected intercept.

Next, we apply the estimated $g(k,T)$ to the relationship expected for $\mu_2(k)$ given by,
\[
\mu_2(k) = \frac{\sigma}{\xi}(g(k,T)-1)+\mu_1
\]
We find good agreement with the theoretical results stated in Theorem \ref{theorem_invariant}. These results are illustrated in figure \ref{fig:plasim}.

\begin{figure}
\includegraphics[width=\textwidth]{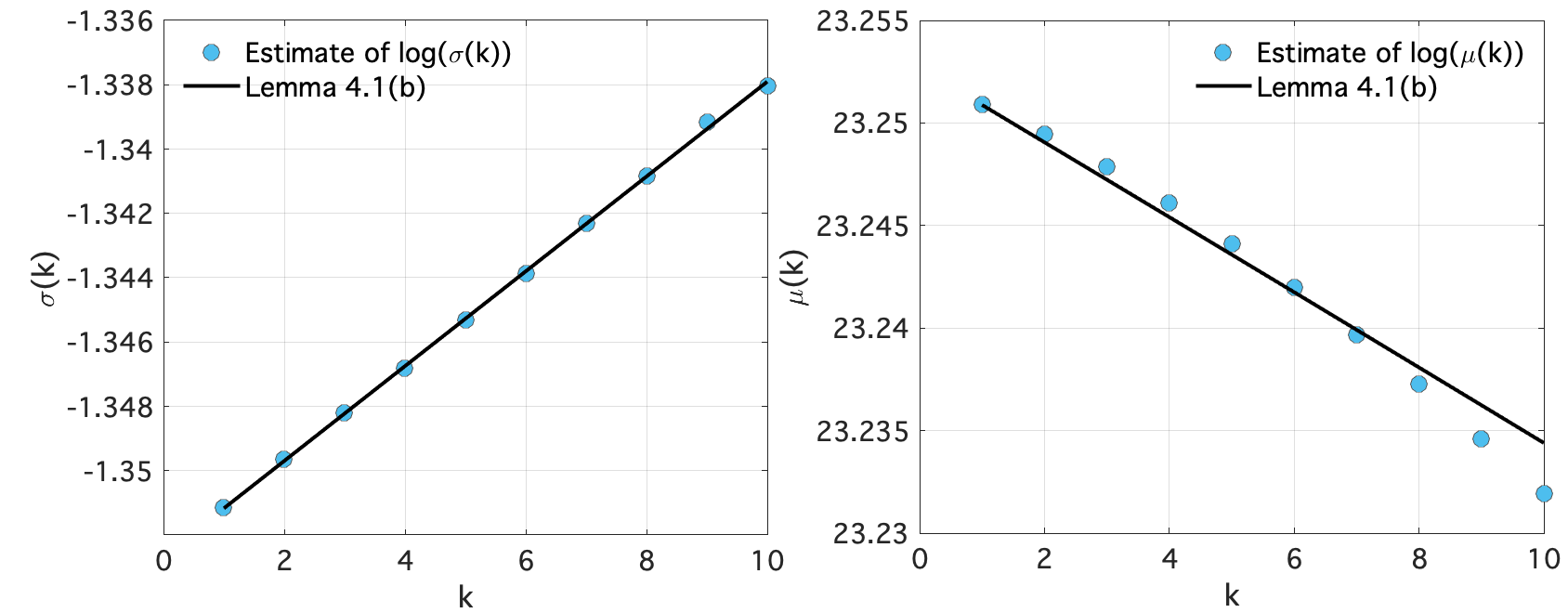}
\caption*{\qquad (a)\qquad \qquad \qquad \qquad \qquad \qquad \qquad \qquad \qquad (b)}
\caption{Plasim generated yearly maximum temperature (a) scale and (b) location relation for $k$-exceedance functional showing good numerical agreement to Theorem \ref{theorem_invariant} (b).}
\label{fig:plasim}
\end{figure}

\subsection{Real-world temperature data (Germany): Climate applications in the Wei-bull case II} In contrast to the rainfall analysis using Fr\'{e}chet results, real-world temperature data has a less evident relationship to the Weibull distance observable explored in the dynamical setting. In addition, daily and annual cycles govern the extremes of temperature and while there are ways of including such non-stationarity into a GEV model, an investigation of this magnitude is beyond the scope of this paper. We therefore focus on the application of the scaling Lemma 4.1 (b) and show that this result holds true for yearly maxima of hourly temperature measured at stations throughout Germany \cite{dwd2}. Recalling that Lemma 4.1 (b) assumes that a scaling $u_n$ and $w_n$ are related by a function $g(k,T)$, depending on the window length $k$ and the system $T$, that the scale parameter $\sigma_2(k)$ is given by,
\[
\sigma_2(k) = g(k,T)\sigma_1\bigg(\frac{\theta_2(k)}{\theta_1}\bigg)^{\xi}\tag{$\ast$}
\]
where $\xi$ is the shape parameter of the GEV for the original time-series (window length $k=1$), $\sigma_2(k)$ is the scale parameter and $\theta_2(k)$ is the extremal index of the time-series with window length $k$. Admittedly, the structure of the function $g(k,T)$ is not known apriori; however, taking the natural logarithm of the relationship $(\ast)$ can allow us to numerically linearize and estimate the $\log(g(k,T))$. The numerical estimate of the function $g(k,T)$ then amounts to fitting a generalized linear model (GLM) on,
\[
\log(g(k,T)) = \log\bigg(\frac{\sigma_2(k)}{\sigma_1}\bigg)-\xi\log\bigg(\frac{\theta_2(k)}{\theta_1}\bigg).
\]
In practice, this also requires reliable estimates of $\theta_2(k)$ and $\sigma_2(k)$ for some sequence of $k = 1:j$ to fit the GLM on $j$ points. As a result, our ability to accurately estimate $\log(g(k,T))$ depends on: the stability of the shape parameter for increasing $k$; the convergence time of $\sigma_2(k)$ for increasing values of $k$, which is likely increasing with increasing $k$ due to lower tail sampling and increased dependence; and the convergence time of the estimate for $\theta_2(k)$. Despite these difficulties, we are able to show that the scaling Lemma 4.1 (b) does indeed describe the behavior we observe in the GEV for the $k$-exceedance functional taken over windows of length $k$.

Calculating the maximum over yearly blocks, we obtain stable maximum likelihood estimates of the shape of the GEV for a sequence of windows of length $k = 1:10$, see figure \ref{fig:temp} (a). Using the maximum likelihood estimates of the scale $\sigma_2(k)$ for $k = 1:10$, and the Ferro-Segers estimate of the extremal index $\theta_2(k)$, we obtain $j = 10$ points with which to estimate the function $\log(g(k,T))$ as a GLM of the data. Results for this estimate are illustrated in figure \ref{fig:temp} (b). 

The GLM representing $\log(g(k,T))$ is then given by some $y(k) = b_0+b_1(k-1)$ so that,
\[
\hat{g}(k,T) = \exp\{b_0+b-1(k-1)\}.
\]
Finally, we use this estimate in the following relationship for the location parameter of the $k$-exceedance functional,
\[
\mu_k = \frac{\sigma_1}{\xi_1}(\hat{g}(k,T)-1)+\mu_1
\]
where $\mu_1$, $\sigma_1$, and $\xi_1$ are the maximum likelihood estimates of the location, scale, and shape, respectively, for the GEV of the yearly maximum of hourly temperature. Estimates from this application of the scaling Lemma 4.1 (b) and the maximum likelihood estimates of the location for increasing sliding windows $k$ of the $k$-exceedance functional are illustrated in figure \ref{fig:temp}. We obtain good numerical agreements with the theoretical relationship.

\begin{figure}
\centering
\includegraphics[width=0.85\textwidth]{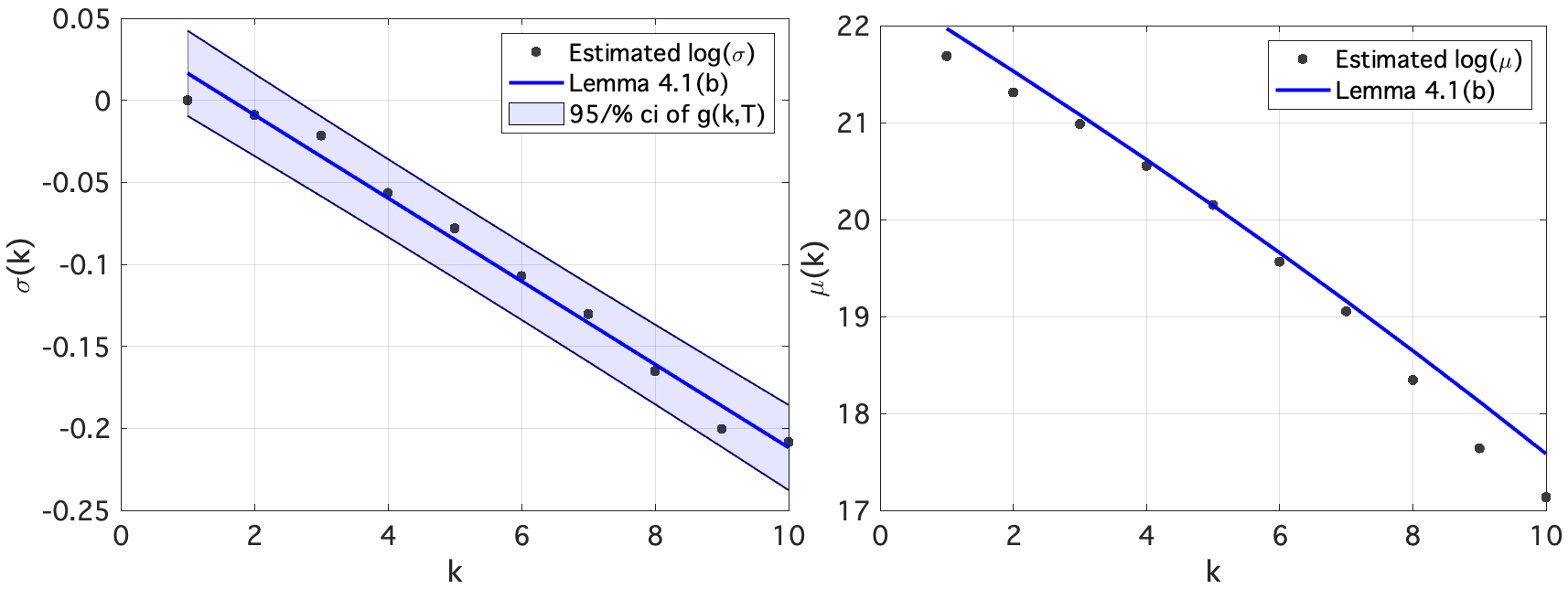}
\caption*{\qquad (a) \qquad \qquad \qquad \qquad \qquad \qquad \qquad \qquad (b)}
\caption{Germany temperature yearly maximum (a) scale and (b) location relation for $k$-exceedance functional showing good numerical agreement to Lemma 4.1(b). Note that $\theta_2(k)$ is not constant in this case, so that results on the trend of the location parameter found in Theorem \ref{theorem_invariant} (b) are not expected to hold.}\label{fig:temp}
\end{figure}

Although the results here support the appropriateness of the scaling Lemma 4.1 (b) to applications in temperature data, the practical use in replacing traditional likelihood estimates with the scale $\sigma$ and location $\mu$ relationships outlined in our lemma needs further investigation. In particular, one would need to consider convergence times of $\sigma$ and $\mu$ for increasing $k$ against the prediction horizon for the GLM representing $g(k,T)$ for large $k$. In any case, the results shown here support the strength of the scaling lemma, particularly in informing the shape parameter $\xi$ which remains constant for any $k$. With such a result, one can also consider keeping $\xi$ constant and optimizing the profile likelihood for the scale, $\sigma$ and location $\mu$ parameters for any value of $k$. We plan to do this investigation in the future for practical data application of these scaling laws. 

\section{Discussion of numerical results}
The implications of these results reach further when one notes that return level plots of either the Generalized Pareto (GP) or the Generalized Extreme Value (GEV) distribution by design, average out the returns of consecutive extremes. In the return level plot for the GP, one must decluster the data by its extremal index prior to fitting a GP and then average the effects of the extremal index back into the model by,
\[
P(X>x) = P(X>u)\bigg[1+\xi\bigg(\frac{x_m-u}{\sigma}\bigg)\bigg]^{-1/\xi} = \frac{1}{m\theta}
\]
where $x_m$ is the return level associated to the return period $m$ with probability $1/m$.

As an example, suppose that a location has an extremal index $\theta = 1/2$ so that, on average, once an exceedance above a threshold $x_m$ is observed, it is followed by another exceedance. Then, for the exceedance threshold $x_m$ corresponding to probability $\frac{1}{365}$, that is once per year, the probability occurring at $x_m$ using the GP distribution after incorporating the extremal index is reduced to $\frac{1}{182.5}$. Although this does translate to two occurrences of daily rainfall exceeding $x_m$ in a single year, it loses all information on the consecutive property of these occurrences.

GEV fitting of the block maxima takes away our ability to interpret results on consecutive over threshold values occurring on the order of the time-series because the blocking method only allows us to keep a single maximum value over the block. As a result, the extremal index gets factored into the location $\mu$ and scale $\sigma$ parameters as stated in the proof of Lemma 4.1 and \cite[Theorem 5.2, Page 96]{Coles} which further reduces any practical information on consecutive extremes in the return level plots for the GEV.

On the other hand, the $k$-exceedance and $k$-average functional, by their structure, preserve information on $k$ consecutive extremes and extremes of $k$ averages, respectively. 

Arguably the most impactful result of this investigation is the unchanging shape parameter observed in the Fr\'{e}chet and Weibull cases for both the $k$-average and $k$-exceedance functionals. Such a result is enough to obtain more accurate estimates on the $k$ consecutive, and $k$ averaged, returns of extreme weather events as estimation of the tail behavior for any $k$ can be done using the original time-series ($k = 1$) which features more tail sampling and hence drastically more accurate estimates of the shape parameter. Some examples on the implications of results in the Fr\'{e}chet case for more accurate real-world return time estimates of consecutive returns of rainfall extremes are provided and discussed for stations throughout Germany. 

Although data from Germany are used for applications of our scaling laws to extreme weather due to the quantity and quality of data available from the Deutscher Wetterdienst, we emphasize that these scaling laws can be applied to a wide range of data. For example, the general Lemma 4.1 (a,b) can be applied to any data having extremes with limiting distributions of the Fr\'{e}chet and Weibull types, respectively, while Theorem \ref{theorem_invariant} can be applied in the non-stationarity setting with appropriate time-dependent adjustments to the GEV parameters. We plan a future numerical investigation using the scaling results introduced in Theorem \ref{theorem_invariant} to non-stationary modeling of successive rainfall extremes using Australian rainfall data which has a known dependence on ENSO. Moreover, more accurate estimates of successive extremes in both the temperature and rainfall setting can be investigated by spatial pooling of data from weather stations using similar techniques as those in the time-dependent GEV setting.

\end{document}